\documentclass[10pt,a4paper,reqno]{amsart} 

\usepackage{amsmath}
\usepackage{mathtools}
\usepackage{amssymb}
\usepackage{graphicx}
\usepackage{color}
\usepackage{latexsym}
\usepackage[utf8]{inputenc}
\usepackage[T1]{fontenc}
\usepackage{comment}
\usepackage{stmaryrd}

\newcommand{\ns}{{\mathbb N}} 
\newcommand{\cs}{{\mathbb C}} 
\newcommand{\rs}{{\mathbb R}} 
\DeclareMathOperator{\Arg}{Arg}

\newcommand{\al}{\alpha}

\newcommand{\eps}{\varepsilon}

\newcommand{\bA}{\bar A}

\newcommand{\E}{\mathbb{E}}
\newcommand{\Var}{\mathbb{V}}

\DeclareMathOperator{\Cyc}{Cyc}

\DeclareMathOperator{\Geom}{Geom}

\DeclareMathOperator{\PD}{PD}

\newcommand{\cA}{\mathcal A}
\newcommand{\cC}{\mathcal C}
\newcommand{\cD}{\mathcal D}

\newcommand{\cG}{\mathcal G}

\newcommand{\cN}{\mathcal N}

\newcommand{\cE}{\mathcal E}

\newcommand{\PP}{\mathbb P}

\newtheorem{Theorem}{Theorem}
\newtheorem{Proposition}[Theorem]{Proposition}

\newtheorem{Example}[Theorem]{Example}

\newtheorem{Lemma}[Theorem]{Lemma}

\newcommand{\beq}{\begin{equation}}
\newcommand{\eeq}{\end{equation}}

\newcommand{\gf}{generating function}
\newcommand{\gfs}{generating functions}

\def\emm#1,{{\em #1}}

\catcode`\@=11
\def\section{\@startsection{section}{1}%
 \z@{.7\linespacing\@plus\linespacing}{.5\linespacing}%
 {\normalfont\bfseries\scshape\centering}}

\def\subsection{\@startsection{subsection}{2}%
  \z@{.5\linespacing\@plus\linespacing}{.5\linespacing}%
  {\normalfont\bfseries\scshape}}

\def\subsubsection{\@startsection{subsubsection}{3}%
 \z@{.5\linespacing\@plus\linespacing}{-.5em}
  {\normalfont\bfseries\itshape}}
\catcode`\@=12

\def\qed{$\hfill{\vrule height 3pt width 5pt depth 2pt}$}

\newcommand{\spacebreak}
{\begin{displaymath} \triangleleft \; \lhd \;
\diamond \; \rhd \; \triangleright
  \end{displaymath}}

\begin{document}
\title
[Asymptotic properties of some minor-closed classes of graphs]
{Asymptotic properties of some\\ minor-closed classes of graphs}

\author[M. Bousquet-M\'elou]{Mireille Bousquet-M\'elou}
\author[K. Weller]{Kerstin Weller}

\address{MBM: CNRS, LaBRI, UMR 5800, Universit\'e de Bordeaux, 
351 cours de la Lib\'eration, 33405 Talence Cedex, France}
\email{mireille.bousquet@labri.fr}
\address{KW: Department of Statistics, University of Oxford, 1 South Parks Road, Oxford, OX1 3TG, United Kingdom}
\email{weller@stats.ox.ac.uk}

%
\thanks{Both authors were
    partially supported by a CNRS-Oxford collaboration scheme
(CNRS/John Fell Oxford University Press Research fund, 2012)}

\subjclass[2000]{}
%

\begin{abstract}
Let $\cA$ be a  minor-closed class of labelled graphs, 
and let  $\cG_{n}$ be a random graph sampled uniformly  from the
set of $n$-vertex  graphs of $\cA$. When $n$ is large, what is
 the probability that $\cG_{n}$ is connected? How many components does it
have?  How large is its biggest component? Thanks
to the work   of McDiarmid   and his collaborators, these questions
are now solved when all excluded minors are 2-connected.

Using exact enumeration, we study
 a collection  of classes $\cA$ excluding non-2-connected minors,
and show that their asymptotic behaviour may be rather different from the 2-connected case.  This behaviour  largely
depends on the nature of dominant 
singularity of the \gf\ $C(z)$ that counts connected graphs of $\cA$. We 
classify our examples accordingly, thus 
taking a first step towards a
 classification  of minor-closed classes of graphs.  
Furthermore, we investigate a parameter that 
has not received any attention in this context yet: the size of the root component. 
It follows non-gaussian limit laws (beta and gamma), and clearly deserves a systematic investigation.
\end{abstract}

\date{\today}
\maketitle

\section{Introduction}
We consider simple graphs on the vertex set $\{1, \ldots, n\}$. A set
of graphs is a \emm class, if it is closed under isomorphisms. A class
of graphs $\cA$ is \emm minor-closed, if any minor\footnote{obtained by contracting or deleting some edges, removing some isolated
vertices and discarding loops and multiple edges} of a graph of $\cA$
 is in $\cA$.  To each such class one can associate its set $\cE$ of
 \emm excluded minors,:  an (unlabelled)  graph is excluded if
 its labelled versions do not belong to $\cA$, but the labelled
 versions of each of its proper minors belong to $\cA$. A remarkable
 result of Robertson and Seymour states that $\cE$ is always
 finite~\cite{robertson-seymour}. We say that the graphs of $\cA$ \emm
 avoid, the graphs of $\cE$. We refer to~\cite{bernardi} for a
 study of the possible \emm growth rates, of
 minor-closed classes.

\begin{figure}[t]
\includegraphics[scale=0.7]{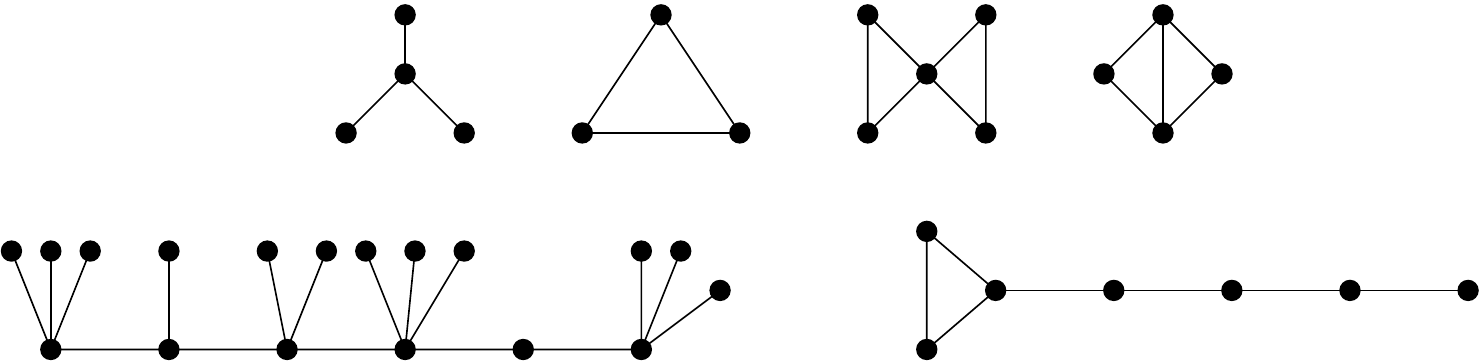}
\caption{A zoo of graphs. Top: the 3-star, the triangle $K_3$, the
  bowtie and the diamond. Bottom: A caterpillar and the 4-spoon
(a $k$-spoon consists of a ``handle'' formed of $k$ edges, to which a
triangle is attached).}  
  \label{fig:ex}
\end{figure}

For a minor-closed class $\cA$, we study the asymptotic properties of
a random graph $\mathcal G_n$ taken uniformly in $\cA_n$, the set of graphs of
$\cA$ having  $n$ vertices: what is the probability
$p_n$ that $\mathcal G_n$ is connected? More generally, what is the number
$N_n$ of connected components? What is the size $S_n$ of the \emm root
component,, that is, the component  containing 1? Or the size $L_n$ of the largest component?

Thanks to the work of McDiarmid and his collaborators, a lot is known
if all excluded graphs are 2-connected:
then $p_n$ converges to a positive
constant (at least $1/\sqrt e$), $N_n$ converges in law to a Poisson
distribution, $n-S_n$ and $n-L_n$ converge in law to the same discrete
distribution. 
Details are given in Section~\ref{sec:2-connected}.

  If some excluded minors are \emm not, 2-connected, the properties of
  $\cG_n$ may be rather different (imagine we exclude
the one edge graph...). This paper takes a
preliminary step towards a classification of the possible behaviours
by presenting  an organized catalogue of examples.

For each class $\cA$ that we study, we
first determine the  \gfs\ $C(z)$ and $A(z)$ that count connected and
general graphs of $\cA$, respectively. The minors that we exclude are
always connected\footnote{We 
refer to~\cite{mcdiarmid-star} for an example where this is not the
case.}, which implies that  $\cA$ is \emm
decomposable, in the sense of Kolchin~\cite{kolchin}: 
a graph belongs to $\cA$
if and only if all its connected components belong to $\cA$. This
implies that $A(z)=\exp(C(z))$. We then derive asymptotic results from
the values of these series. They are illustrated throughout the paper
by pictures of large random graphs, generated using \emm Boltzmann
samplers,~\cite{boltzmann}.  Under a Boltzmann distribution, two
graphs of $\cA$ having the same size always
 have the same probability. The most difficult class we study is that
 of graphs avoiding the bowtie (shown in Figure~\ref{fig:ex}).

Our   results make extensive use of the techniques of Flajolet and Sedgewick's book~\cite{flajolet-sedgewick}:
symbolic combinatorics, singularity analysis, saddle point method, and their application to
the derivation of limit laws. We recall a few basic principles in
Section~\ref{sec:gf}. We also need and prove two general results
of independent interest related to the saddle point method or, more
precisely, to Hayman admissibility
(Theorems~\ref{thm:h-h} and~\ref{thm:ext-adm}).

Our results are summarized in Table~\ref{table}. A first principle
seems to emerge: 
\begin{quote}
\emph {the more rapidly $C(z)$ diverges at its
radius of convergence $\rho$, the more components $\cG_n$ has, and the
smaller they are.}
\end{quote}
In particular, when $C(\rho)$ converges, then the properties of $\cG_n$ are qualitatively
the same as in the 2-connected case (for which $C(\rho)$ always converges~\cite{mcdiarmid}), except that the limit of $p_n$ can
be arbitrarily small.
When $C(\rho)$
diverges, a whole variety of  behaviours can be observed,
depending on the nature of the singularity of $C(z)$ at $\rho$:
the probability $p_n$
tends always to $0$, but at
various speeds; the number $N_n$ 
of components goes to infinity at various speeds (but is invariably gaussian after normalization); the
size $S_n$ of the root component and the size $L_n$ of the largest
component follow, after normalization, non-gaussian
limit laws: for instance, a Gamma or Beta law for $S_n$, and for $L_n$
a Gumbel law or the first component of a Poisson-Dirichlet
distribution. 
Cases where $C(z)$ converges, or diverges at most logarithmically, are
addressed using singularity analysis (Sections~\ref{sec:trees}
and~\ref{sec:b+d}), while those in which $C(z)$ diverges faster (in practise,
with an algebraic singularity) are addressed with the saddle point
method (Sections~\ref{sec:bounded}
to~\ref{sec:bowtie}). Section~\ref{sec:tools} gathers general results on the saddle
point method and Hayman admissibility.

\begin{table}[b!]
  \centering
  \begin{tabular}{|c|c|c|c|c|c|c|c|}
\hline
    Excluded &  $C(\rho)$ & Sing. &
$\lim p_n 
$ & number $N_n$& root 
& largest  & Refs. and 
\\
    minors &    &  of  $C(z)$&
& of  comp.&  comp. $S_n$
&   comp. $L_n$ & methods
\\
\hline\hline
2-connected & $< \infty$ & ? & $\ge 1/\sqrt e$ & $O(1)$ & $n-S_n$&$n-L_n$&\cite{berry,mcdiarmid,MR2117936,MR2249274}\\
             &            & &$<1$ & Poisson & $\rightarrow $
disc. & $\rightarrow$ disc.&
Sec.~\ref{sec:2-connected}\\
\hline
at least& $< \infty$ & $(1-ze)^{3/2}$ &$>0$& id. & id. & id.
& Sec.~\ref{sec:trees}
\\
 a spoon,&&&$\le 1/\sqrt e$ &&&& sing. an.\\
but no tree&&&&&&&\\
\hline
\hline
\includegraphics[scale=0.7]{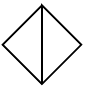}\hskip 4mm 
\includegraphics[scale=0.7]{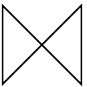}& $\infty$ & $\log$ & 0 & $\log n$
& $n$& $\PD^{(1)}(1/4)$ & Sec.~\ref{sec:b+d}
\\
                && ($+\sqrt{{} \hskip 2mm {}}$)& & gaussian
                &$\frac 1 4(1-x)^{-3/4}$ && sing. an.
\\
\hline\hline
\includegraphics[scale=0.7]{bowtie}& $\infty$ & $1/\sqrt{{} \hskip 2mm {}}$
& 0 & $n^{1/3}$  &  $n^{2/3}$ &? &Sec.~\ref{sec:bowtie}
\\
&&&& gaussian  & $2\sqrt{x/\pi} e^{-x}$  && saddle
\\
\hline\hline
\includegraphics[scale=0.7]{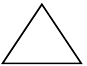}\hskip 4mm 
\includegraphics[scale=0.7]{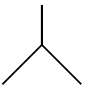}& $\infty$& simple&0&$\sqrt
n$&$\sqrt n$& $\sqrt n \log n$& Sec.~\ref{sec:pole}
\\
  (path forests)                        &      & pole && gaussian&$xe^{-x}$
                                   & Gumbel& saddle
\\
&&&&&&& \\
\hline
\includegraphics[scale=0.7]{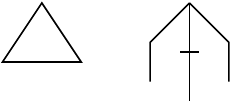} & $\infty$ & id. & 0 & id. & id. &? & Sec.~\ref{sec:pole}
\\
(forests of &&&&&&& saddle
\\
caterpillars) &&&&&&&
\\
\hline
\includegraphics[scale=0.7]{3star} &$\infty$  & id. & 0 & id. & id. &
? &Sec.~\ref{sec:pole-log}
\\
(max. deg. 2)& & ($+ \log$) &&&&& saddle
\\
\hline
\hline
&&&&&&&\\
all conn. graphs& $\infty$ & entire & $0$ & $n/k$ & $k$ & $k$ & Sec.~\ref{sec:bounded}
\\
of size $k+1$     &&(polynomial)&& gaussian &Dirac & Dirac& saddle\\
\hline
  \end{tabular}
\vskip 3mm
  \caption{Summary of the results: for each quantity $N_n$, $S_n$ and
    $L_n$, we give the order of the expected value (up to a
    multiplicative constant, except in the last line where constants
    are exact) and a  description (name or  density) of the limit
    law. The examples are ordered according to the speed of divergence
    of $C(z)$ near its radius  $\rho$. \emm Spoons, are defined in
    Figure~\ref{fig:ex}. As we get lower in the table,
    the graphs have more components, of a smaller size. The symbol
    $\PD^{(1)}(1/4)$ stands for the first component of a
    Poisson-Dirichlet distribution of parameter $1/4$.
}
  \label{table}
\end{table}

\medskip
Let us conclude with a few words on the size of the
root component. It appears that this parameter, which can be defined
for any exponential family of objects, has not been studied
systematically yet, and follows interesting (\emm i.e.,, non-gaussian!) continuous limit laws, after
normalization. In an independent paper~\cite{porte}, we perform such a
systematic study, in the spirit of what Bell \emm et al,.~\cite{bell} or Gourdon~\cite{gourdon} did
for the number of components or
 the largest component, respectively.
This project is also reminiscent of the study of the 2-connected
component containing the root vertex in a planar map, which
also leads to 
a non-gaussian  continuous limit law, namely an Airy distribution~\cite{banderier-maps}. 
This distribution is also related to the size of the largest 2- and
3-connected components in various classes of graphs~\cite{gimenez-given}.

\section{``Generatingfunctionology'' for graphs}\label{sec:gf}
Let $\cE$ be a finite set of (unlabelled)  \emm connected, graphs that
forms an antichain for the minor order (this means that no graph of
$\cE$ is a minor of another one). Let $\cA$ be the set of labelled
graphs  that do not contain any element of $\cE$ as a minor. We denote
by $\cA_n$ the subset of $\cA$ formed of graphs having $n$ vertices
(or \emm size $n$), and by $a_n$ the cardinality of $\cA_n$. 
The associated  exponential \gf\  is $A(z)=\sum_{n\ge 0} a_n z^n/n!$. We use
similar notation ($c_n$ and $C(z)$) for the subset $\cC$ of $\cA$
consisting of (non-empty) connected graphs. Since the excluded minors are
connected, $\cA$ is decomposable, and
$$
A(z)= \exp( C(z)).
$$
Several refinements of this series are of interest, for instance the
\gf\ that keeps  track of the number of (connected) components as well:
$$
A(z,u)=
 \sum_{G \in \cA} u^{c(G)} \frac{z^{|G|}}{|G|!},
$$
where 
$|G|$ is  the size of  $G$ 
and $c(G)$  the number of its components.
Of course,
$$
A(z,u)= \exp(u C(z)).
$$
We denote by $\cG_n$ a uniform random graph of $\cA_n$, and by $N_n$
the number of its components. Clearly,
\beq\label{Nni}
\PP(N_n=i)= \frac{[z^n] C(z)^i}{i! [z^n] A(z)},
\eeq
where $[z^n] F(z)$ denotes the coefficient of $z^n$ in the series
$F(z)$.
 The $i$th factorial moment of $N_n$ is
$$
\E(N_n(N_n-1)\cdots (N_n-i+1))= 
\frac{[z^{n}] \frac{\partial ^i  A}{\partial u^i}(z,1)}
{{[z^{n}] A(z)}}
= 
\frac{[z^{n}] C(z)^iA(z)}
{{[z^{n}] A(z)}}.
$$
Several general results  provide a limit law for $N_n$ if
$C(z)$ satisfies certain conditions: for instance the
results of Bell \emm et
al,.~\cite{bell} that require $C(z)$ to converge at its radius of convergence; or the \emm exp-log
schema, of~\cite[Prop.~IX.14, p.~670]{flajolet-sedgewick}, which
requires $C(z)$ to diverge with a logarithmic singularity (see also
the closely related results of~\cite{arratia} on logarithmic structures). 
We  use these results when applicable, and prove a new
result of this type, based on Drmota \emm et al,.'s notion of \emm extended
Hayman admissibility,, which applies when $C(z)$ diverges with an
algebraic singularity. We believe it to be of independent interest
(Theorem~\ref{thm:ext-adm}). 

\medskip
We also study the  size $c_1$ of the \emm root component,,
which is the component containing
the vertex~1. We  define accordingly
$$
\bA(z,v)=\sum_{G \in \cA, G\not = \emptyset} v^{c_1(G)-1} \frac{z^{|G|-1}}{(|G|-1)!}.
$$
The choice of  $|G|-1$ instead of $|G|$  simplifies
slightly some calculations.  Note that 
$\bA(z,1)= A'(z)= C'(z) A(z).$ 
Denoting by $S_n$  the size of the root component in $\cG_n$,
we have
\beq\label{local}
\PP(S_n=k)=\frac{c_{k}a_{n-k}{n-1 \choose {k-1}}}{a_n}= \frac k n
\frac{c_k}{k!} \frac{a_{n-k}}{(n-k)!} \frac{n!}{a_n}.
\eeq
Equivalently,  the series $\bA(z,v)$ is given by
\beq\label{global}
\bA(z,v)
= C'(zv) A(z).
\eeq
 The $i$th factorial moment of $S_n-1$ is
\beq\label{S-moments}
\E((S_n-1)\cdots (S_n-i))= 
\frac{[z^{n-1}] \frac{\partial ^i  \bar A}{\partial v^i}(z,1)}
{{[z^{n-1}] \bar A(z,1)}}
= 
\frac{[z^{n-i-1}] C^{(i+1)}(z)A(z)}
{{n[z^{n}] A(z)}}.
\eeq
Surprisingly, this parameter has not been studied
before. Our examples give rise to
non-gaussian limit laws (Beta or Gamma,
cf. Propositions~\ref{prop:size-b+d} or~\ref{prop:pole-root}). In
fact, the form~\eqref{global} of the \gf\ 
shows that this parameter is bound to give rise to interesting limit
laws, as both the location and nature of the singularity change as
$v$ moves from $1-\varepsilon$ to $1+\varepsilon$. Using the terminology
of Flajolet and Sedgewick~\cite[Sec.~IX.11]{flajolet-sedgewick},
 a \emm phase transition, occurs. We are currently working on
a systematic study of this parameter
in exponential structures~\cite{porte}.

\medskip
Finally, we denote by $C^{[k]}(z)$ the  \gf\ of connected graphs of
$\cA$ of size less than $k$:
$$
C^{[k]}(z)= \sum_{n=1}^{k-1} c_n \frac{z^n}{n!},
$$
and study, for some classes of graphs,
the size $L_n$ of the largest component. We have
\beq\label{df}
\PP(L_n < k) = \frac{[z^n] \exp(C^{[k]}(z))}{[z^n]A(z)}.
\eeq

\bigskip
We use in this paper two main methods for studying the asymptotic
behaviour of a sequence $(a_n)_n$ given by its \gf\
$A(z)$. The first one is the \emm singularity analysis,
of~\cite[Chap.~VI]{flajolet-sedgewick}. Let us describe briefly how it
applies, for the readers who would not be familiar with it. Assume that
$A(z)$ has a unique singularity of minimal modulus (also called \emm
dominant,) at its radius of
convergence~$\rho$, and is  analytic in a \emm
$\Delta$-domain,, that is, a domain of the form
$$
\{z :|z|<r, z\not = \rho \hbox{ and } |\Arg (z-\rho)|> \phi\}
$$
for some $r>\rho$ and $\phi \in (0, \pi/2)$.
Assume finally that, as $z$ approaches $\rho$ in this domain,
$$
A(z)=S(z)+ O(R(z)),
$$
where $S(z)$ and $R(z)$ are functions belonging to the simple
\emm algebraic-logarithmic scale, of~\cite[Sec.~VI.2]{flajolet-sedgewick}. Then one can \emm
transfer, the above singular estimates for the series into asymptotic
estimates for the coefficients:
$$
[z^n]A(z)= [z^n]S(z) + O([z^n] R(z)).
$$
Since $S$ and $R$ are simple functions, the asymptotic behaviour of
their coefficients is well known, and the estimate of $[z^n]A(z)$ is
thus explicit. We use singularity analysis in
Sections~\ref{sec:2-connected} to~\ref{sec:b+d}. The second method we
use is the \emm saddle point method,. In Section~\ref{sec:tools} we recall how to
apply it, and then use it in Sections~\ref{sec:bounded} to~\ref{sec:bowtie}.

When dealing directly with sequences rather than \gfs, a useful notion
will be that of \emm smoothness,: the sequence $(f_n)_{n\ge 
  0}$ is {\em smooth} if  $f_{n-1}/f_n $ converges as  $n$ grows. The
limit is then the radius of convergence of the series $\sum_n f_n z^n$.
%

\section{Classes defined by 2-connected excluded minors}
\label{sec:2-connected}
We assume in this section that  at least one minor is excluded,
and that all excluded minors are 2-connected. This includes 
the classes of forests,  series-parallel graphs,  outer-planar
graphs, planar graphs... Many results are known in
this case.  We recall briefly some of them, 
and state a new (but easy) result dealing with the size of the root
component. The general picture is that the class $\cA$ shares many 
properties with the class of forests. 
\begin{Proposition}[{\bf The number of graphs --- when excluded minors are
  2-connected}]
\label{prop:number-2conn}
The \gfs\ $C(z)$ and $A(z)=e^{C(z)}$ are finite at their (positive) radius of
convergence~$\rho$. 
Moreover, the sequence $(a_n/n!)_n$ is  smooth.

The probability that  $\cG_n$ is connected tends to
$1/A(\rho)$, which is clearly in $(0,1)$.  In fact, this limit is also
larger than or 
equal to $1/\sqrt e$. The latter value is reached when $\cA$ is the
class of forests.
\end{Proposition}
The fact that $\rho$ is positive is due to Norine \emm et
al,.~\cite{norine}, and holds for any proper minor-closed class. The next results are due to
McDiarmid~\cite{mcdiarmid} (see also the earlier
papers~\cite{MR2117936,MR2249274}). The fact that $1/A(\rho)\ge 1/\sqrt
e$, or equivalently, that $C(\rho)\le 1/2$, was conjectured
in~\cite{MR2249274}, and then proved independently in~\cite{berry} and~\cite{kang}. 

\begin{figure}[htb]
\includegraphics[height=10cm,width=8cm,angle=90]{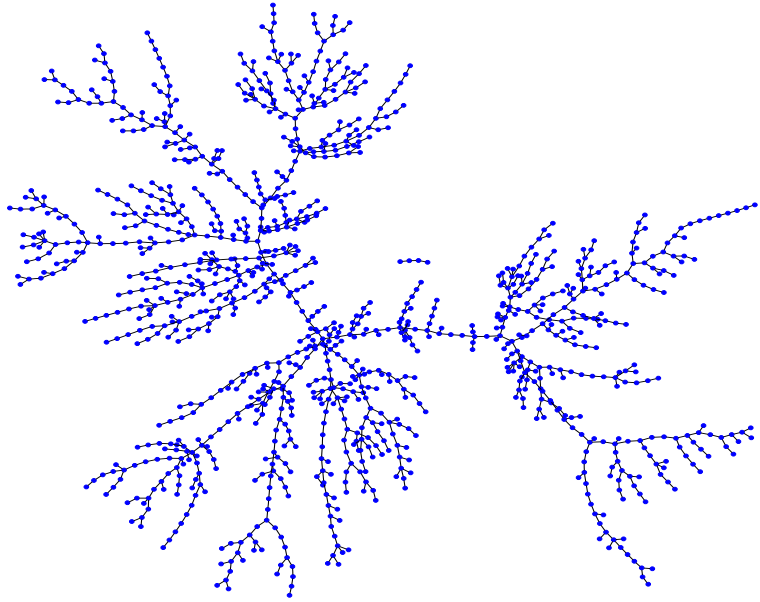}
\caption{A random forest of size $n=1165$.
It has two connected components.}
  \label{fig:forest}
\end{figure}

\begin{Example}\label{ex:forests}
 {\rm A basic, but important example is that of forests, illustrated
   in Figure~\ref{fig:forest}. We have in this case
$$
C(z)=T(z)-\frac{T(z)^2}2,
$$
where $T(z)=ze^{T(z)}$ counts rooted trees (see for
instance~\cite[p.~132]{flajolet-sedgewick}). The series  $T$, $C$ and
$A=e^C$ have radius  of convergence
$\rho=1/e$, with the following singular expansions at this point:
\beq\label{Tsing}
\begin{array}{lllllllllllcccccccccccccccccccc}
  T(z)=1&-\sqrt 2 ({1-ze})^{{1/2}}&+ \frac 2 3 (1-ze)
  &-\frac{11\sqrt 2}{36}(1-ze)^{{3/2}}&+ O((1-ze)^{2}), \\
C(z)=\frac 1 2&&-(1-ze)&+ \frac{2\sqrt 2}3
(1-ze)^{{3/2}} &+ O( (1-ze)^2),
\\
A(z)= \sqrt e  &&- \sqrt e(1-ze) &+ \sqrt e\frac{2\sqrt 2}3
(1-ze)^{{3/2}} &+ O( (1-ze)^2).
\end{array}
\eeq
The  singularity analysis of~\cite[Chap.~VI]{flajolet-sedgewick}
applies: the three series are  analytic in a $\Delta$-domain, and
their coefficients satisfy
$$
t_n \sim n! \frac {e^n}{\sqrt{2\pi}n^{3/2}},\quad 
{c_n \sim n! \frac {e^n}{\sqrt{2\pi}n^{5/2}}},\quad \hbox{and} \quad 
{a_n \sim  \sqrt e\, c_n}.
$$

We will also consider rooted trees of height less than $k$ (where
by convention the  tree consisting of a single vertex has height
$0$). Let $T_k(z)$ denote their \gf. Then $T_1(z)=z$ and  for $k \ge 1$,
$$
T_{k+1}(z)=z e^{T_{k}(z)}.
$$
Note that $T_k(z)$ is entire.}
\end{Example}

\noindent
{\bf Note.} When all excluded minors are 2-connected, $C(\rho)$ always
converges, but 
the nature of the singularity of $C(z)$ at $\rho$
depends on the class: it is for instance
$(1-z/\rho)^{3/2}$ for forests (and more generally, 
for \emm subcritical, 
classes~\cite{fusy-subcritical}),
but  $(1-z/\rho)^{5/2}$ for planar graphs. We refer to~\cite{gimenez-given} for a
more detailed discussion that applies
to classes that exclude 3-connected minors.

\medskip
\begin{Proposition}
  [{\bf Number of components --- when excluded minors are
  2-connected}]
The mean  of $N_n$ satisfies:
\[
\E(N_n) \sim 1+C(\rho)
\]
 and the random variable $N_n-1$ converges in law to  a Poisson
 distribution of parameter $C(\rho)$. That is, as  $n\rightarrow \infty$,
\beq\label{poisson-def}
\PP(N_n=i+1) \rightarrow  \frac{C(\rho)^i}{i!e^{C(\rho)}}.
\eeq
\end{Proposition}
We refer  to~\cite[Cor.~1.6]{mcdiarmid} for a proof.
The largest component is known to contain almost all vertices, and it
is not hard to prove that the same holds  for the root component. In
fact, the tails of the random variables $S_n$ and $L_n$ are  related
by the following simple result.
\begin{Lemma}\label{lem:SL}
  For any class of graphs $\cA$, and $k<n/2$,
$$ \PP(S_n=n-k)= \frac{n-k}n\, \PP(L_n=n-k).
$$
\end{Lemma}
\begin{proof}
Let us denote by $B_n$
the (lexicographically first)
biggest component of  $\cG_n$. Its size is thus
$L_n$. We have, for $n>2k$,
\begin{eqnarray*}
  \PP(S_n=n-k)&=& \PP(S_n=n-k \hbox{ and } 1 \in B_n) + \PP(S_n=n-k \hbox{
  and } 1 \not \in B_n)
\\
&=&\PP(L_n=n-k \hbox{ and } 1 \in B_n)+ \PP(S_n=n-k \hbox{
  and } 1 \not \in B_n)
\\
&=&\PP(1 \in B_n |L_n=n-k)\PP(L_n=n-k)+ \PP(S_n=n-k \hbox{
  and } 1 \not \in B_n)
\\
&=&\frac{n-k}n\, \PP(L_n=n-k).
\end{eqnarray*}
Indeed, there cannot be two
components of size $n-k$ or more. This implies
that $\PP(S_n=n-k \hbox{  and } 1 \not \in B_n)=0$.
\end{proof}

\begin{Proposition}[{\bf The root component and the largest component --- when excluded minors are  2-connected}]
\label{prop:largest-2conn}
The random variables $n-S_n$ and $n-L_n$ both converge to a
discrete limit distribution $X$ given by
$$
\PP(X=k)=\frac 1{A(\rho)} \frac{a_k \rho^k}{k!}.
$$
\end{Proposition}
\begin{proof} By
  Lemma~\ref{lem:SL}, the two statements are equivalent. The $L_n$  result has been proved by
  McDiarmid~\cite[Cor.~1.6]{mcdiarmid}. 

We give an independent proof (of the $S_n$ result), as we
will recycle its ingredients later for certain classes of graphs that avoid
non-2-connected minors.
Let $k\ge 0$ be fixed.  By~\eqref{local}, 
$$
\PP(S_n=n-k)=\frac{c_{n-k}a_k{n-1 \choose k}}{a_n}= \frac{a_k}{k!}\ 
\frac{c_{n-k}}{a_{n-k}}\  \frac{(n-1)! a_{n-k}}{(n-k-1)! a_n}.
$$
By Proposition~\ref{prop:number-2conn}, the term ${c_{n-k}}/{a_{n-k}}$, which is the probability
that a graph of size $n-k$ is connected, converges to
$1/A(\rho)$. Moreover,  the sequence $a_n/n!$ is  smooth,
so that $\frac{(n-1)! a_{n-k}}{(n-k-1)! a_n}$ converges to
$\rho^k$. The result follows.
\end{proof}

 A more precise result is actually available. Let us  call \emm fragment,  the union of the 
 components that differ from the biggest component $B_n$. Then McDiarmid describes the limit law of the fragment,
not only of his size~\cite[Thm.~1.5]{mcdiarmid}:
the probability that the fragment is isomorphic to a 
given unlabelled graph $H$ of size $k$ is 
$$
\frac 1 {A(\rho)}\frac{\rho^k}{{\rm aut}(H)},
$$
where ${\rm aut}(H)$ is the number of automorphisms of $H$.

\section{When trees dominate: $\boldsymbol {C(z)}$ converges at $\boldsymbol {\rho}$}
\label{sec:trees}

Let $\cA$ be a decomposable class of graphs 
(for instance, a class defined by   excluding connected minors),  
satisfying the following conditions:
\begin{enumerate}
\item $\cA$ includes all trees,
\item the \gf\ $D(z)$ that counts the connected graphs of $\cA$ that
  are \emm not, trees has radius of convergence (strictly) larger than $1/e$
  (which is the radius of trees).
\end{enumerate}
We then say that $\cA$ is \emm dominated by trees,. Some examples
are presented below. In this case,
the properties that hold for forests (Section~ \ref{sec:2-connected})
still hold, except that the probability   
$c_n/a_n$ that  $\cG_n$ is connected tends to a
limit that is now \emm at most, $1/\sqrt e$.
We will see that this limit can become arbitrarily small.

\begin{Proposition}[{\bf The number of graphs --- when trees dominate}]
\label{prop:trees-dom}
Let $T(z)$ be the \gf\ of rooted trees, given by $T(z)=ze^{T(z)}$.
 Write  the \gf\ of connected graphs in the class  $\cA$ as
$$
C(z)= T(z)-\frac {T(z)^2} 2 + D(z).
$$
 The \gf\ of graphs of $\cA$ is
$
A(z)=e^{C(z)}.
$
  As $n\rightarrow \infty$,
$$
c_n \sim n! \frac {e^n}{\sqrt{2\pi}n^{5/2}}\quad \hbox{and} \quad 
 a_n \sim A(1/e)c_n.
$$
In particular, the probability that $\cG_{n}$ is
connected tends to 
$1/A(1/e)=e^{-1/2-D(1/e)}$  as $n\rightarrow \infty$. 
 \end{Proposition}
 \begin{proof}
As in Example~\ref{ex:forests}, we use singularity analysis~\cite[Chap.~VI]{flajolet-sedgewick}. 
By assumption, $D(z)$ has radius of convergence  larger than $1/e$, and the
singular behaviour of $C(z)$ is that of unrooted trees.
More precisely, it follows from~\eqref{Tsing} that, as $z$ approaches
$1/e$, 
$$
C(z)= 1/2+ D(1/e)-(1-ze)(1+D'(1/e)/e)+ \frac{2\sqrt 2}3
(1-ze)^{3/2} + O( (1-ze)^2),
$$
 this expansion being valid in a $\Delta$-domain. This gives the estimate of
$c_n$ via singularity analysis.
For the series $A$, we find
$$
A(z) = e^{1/2+D(1/e)} \left( 1- (1-ze)(1+D'(1/e)/e) + \frac{2\sqrt 2}3
(1-ze)^{3/2} + O( (1-ze)^2)\right),
$$
and the estimate of $a_n$ follows.
\end{proof}

\begin{Proposition}[{\bf Number of components --- when trees dominate}]
\label{prop:number-trees-dom}The mean  of $N_n$ satisfies:
\[
\E(N_n) \sim 1+C(1/e)
\]
 and  $N_n-1$ converges in law to  a Poisson
 distribution of parameter $C(1/e)$ (see~\eqref{poisson-def}).
\end{Proposition} 
\begin{proof} 
We can start from~\eqref{Nni} and apply singularity analysis. Or we can apply a ready-to-use result of Bell
\emm et al,.~\cite[Thm.~2]{bell}, which uses the facts (proved in
Proposition~\ref{prop:trees-dom}) that the sequences $n c_{n-1}/c_n$ and $c_n/a_n$
converge.
\end{proof}

\begin{Proposition}[{\bf Size of  components --- 
when trees dominate}]\label{prop:size-trees-dom}
 The random variable $n-S_n$ converges to a
discrete limit distribution $X$ given by
$$
 \PP(X=k)=\frac 1{A(1/e)} \frac{a_k e^{-k}}{k!},
$$
where $a_k$ and $A(z)$ are given in
Proposition~{\rm\ref{prop:trees-dom}}. The same holds for $n-L_n$.
\end{Proposition}
\begin{proof} 
The two ingredients used in the proof of
Proposition~\ref{prop:largest-2conn} to establish the limit law of
$n-S_n$  (namely, smoothness of $a_n/n!$ and
convergence of $c_n/a_n$), still hold here (see Proposition~\ref{prop:trees-dom}).  Lemma~\ref{lem:SL} gives
then  the law of $n-L_n$.
\end{proof}

\medskip
We now present a collection  of classes  dominated by trees.

\begin{Proposition}\label{prop:spoon}
  Let $k\ge 1$. Let $\cA$ be a decomposable class of graphs that 
includes
  all trees, and such that all graphs of $\cA$ avoid the  $k$-spoon
  (shown in Figure~{\rm\ref{fig:ex}}).
Then $\cA$ is dominated by trees, and the results of
Propositions~{\rm\ref{prop:trees-dom}},~{\rm\ref{prop:number-trees-dom}} and
{\rm\ref{prop:size-trees-dom}} hold.
\end{Proposition}
\begin{proof}
 Clearly, it suffices to prove this proposition when $\cA$ is
  exactly the class of graphs avoiding the $k$-spoon, which we
  henceforth  assume.  

We partition the set $\cC$ of connected graphs of $\cA$ into three
  subsets: the set $\cC_0$ of trees, counted by $C_0=T-T^2/2$ with
  $T\equiv T(z)$, the set
  $\cC_1$ of unicyclic graphs (counted by $C_1$), and finally the set $\cC_2$
  containing graphs with at
  least two cycles (counted by $C_2$). Hence $C=T-T^2/2+ C_1 + C_2$. We will
  prove that $C_1$ has radius  of convergence
(strictly) 
larger than $1/e$, and that $C_2$ is  entire.

A unicyclic graph belongs to $\cC$ if
and only if all trees attached to its unique 
cycle have height less than $k$.   The \gf\ of cycles is given by:
\beq\label{bt1}
\Cyc(z)= \frac 1 2 \sum_{n\ge 3} \frac{z^n}n= \frac 1 2 \left(\log
\frac 1{1-z} - z  - \frac {z^2}2\right).
\eeq 
Hence, the basic rules of the \emm symbolic method,
of~\cite[Chap.~II]{flajolet-sedgewick} give:
\beq\label{Cyc-Tk}
C_1(z)=\Cyc(T_{k})=
 \frac 1 2 \left (\log \frac 1 {1-T_{k}(z)}
-T_{k}(z) -\frac{T_{k}(z)^2}2\right),
\eeq
where $T_k$ counts rooted trees of height less than $k$ and is given
in Example~\ref{ex:forests}.
Recall from this example that $T(z)$ equals 1 at its unique dominant
singularity $1/e$. Also, $T_{k}(z)<T(z)$ for all  $z\in [0, 1/e]$ since $T_{k}$
counts fewer trees than $T$. In particular, $T_{k}(1/e)<1$ and   $C_1(z)$  has radius of
convergence  larger than $1/e$.

We now want to prove that $C_2$ is entire.   The \emm
 $(2)$-core, of a connected graph $H$ is the (possibly empty) unique maximal
 subgraph of minimum degree 2.
 It can be obtained from $H$ by deleting recursively all vertices  of
degree 0 or 1
 (or, in a non-recursive fashion, all dangling trees of $H$). 
By extension, we call \emm core, any  connected graph of minimum degree 2.
Let $\bar \cC_2$ denote the
set of cores having several cycles and avoiding the
$k$-spoon, and $\bar C_2$ the associated \gf.  The inequality
$$
C_2(z) \le \bar C_2(T_{k}(z))
$$
holds, coefficient by coefficient, because the core of a graph of
$\cC_2$ has several cycles and avoids the $k$-spoon. Since $T_{k}$ is entire, it
suffices to prove that $\bar C_2$ is entire. It follows
from~\cite[Thm.~3.1]{bernardi} that it suffices to prove
that no graph $G$ of $\bar \cC_2$
contains a path of length $3k-1$. So let $P=(v_0, v_1, \ldots, v_\ell)$ be a
path of maximal length in $G$, and assume that $\ell \ge 3k-1$. We will
prove that $G$ contains the $k$-spoon as a minor. Since $P$ is
maximal and $G$ is a core, there exist $v_i$ and $v_j$, with $i\ge 2$ and $j\le \ell-2$,
such that the edges $\{v_0, v_i\}$ and $\{v_j, v_\ell\}$ belong to
$G$.

If $i=\ell$ or $j=0$, let  $\bar P$ be the cycle of $G$ formed of $P$
and the edge $\{v_0, v_\ell\}$. Let $\bar Q$ be  another cycle of
$G$. If $\bar Q$ contains at most one 
vertex of $P$ (Figure~\ref{fig:spoon}(a)), we find an $\ell$-spoon by
deleting one edge of $\bar P$, contracting $\bar Q$ into a
3-cycle and 
one of the paths joining $P$ to $\bar Q$ 
into a point. If $\bar Q$ contains at least two
vertices $v_a$ and $v_b$ 
of $P$, with $a<b$ (Figure~\ref{fig:spoon}(b)), we may assume that $\bar Q$  consists of the edges
$\{v_a,v_{a+1}\}, \ldots, \{v_{b-1},v_b\}$ and of a path $Q$ that
only meets $P$ at $v_a$ and $v_b$. Let
$\bar R$ denote the cycle formed of the path $Q$ and the path $(v_b,
v_{b+1}, \ldots, v_\ell, v_0, \ldots,  v_a\}$. Then we obtain a
$p$-spoon, with $p\ge \lceil 3k/2\rceil -1 \ge k$,  by contracting the
shortest of the cycles $\bar Q$ and $\bar R$ into a 3-cycle and deleting an edge
ending at $v_a$ from the other.

\begin{figure}[t!]
\scalebox{0.7}{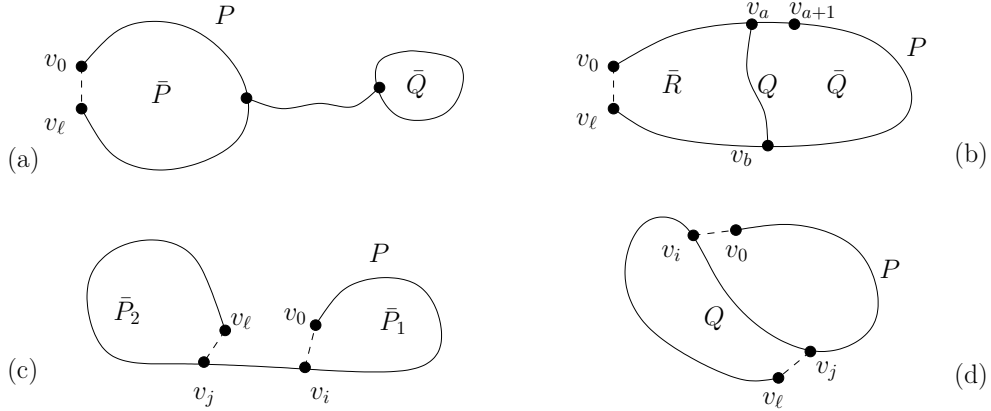}
\caption{A core having several cycles and avoiding the $k$-spoon
  cannot contain a path of length   $3k-1$.}
\label{fig:spoon}
\end{figure}

Assume now that $i<\ell$ and $j>0$. 
Suppose first that  $i\le j$ (Figure~\ref{fig:spoon}(c)).
By symmetry, we may assume that  the cycle
$\bar P_1=(v_0, \ldots, v_i)$ is shorter than (or equal in length to) the cycle
$\bar P_2=(v_j, \ldots, v_\ell)$. 
In particular, $i\le \ell/2$. Contract $\bar P_1$ into a 3-cycle, and remove
the edge $\{v_j,v_\ell\}$ from $\bar P_2$: this gives a $p$-spoon with
$p=\ell-i\ge \lceil \ell/2\rceil \ge k$. Assume now that $j<i$ (Figure~\ref{fig:spoon}(d)). Consider the three
following paths joining $v_i$ and $v_j$:  $(v_i, v_{i-1}, \ldots,
v_j)$, $(v_i, v_0, v_1,  \ldots, v_j)$ and $(v_i, v_{i+1}, \ldots,
v_\ell, v_j)$. Since the sum of the lengths of these
paths is $\ell+2\ge 3k+1$, one of them, say $(v_i, v_0, v_1,  \ldots, v_j)$,
has length at least $k+1$. That is, $j \ge k$. Delete from this path the edge
$\{v_i, v_0\}$, and contract the cycle formed by the other two
paths into a 3-cycle: this gives a $j$-spoon, with $j\ge k$.
\end{proof}

The simplest non-trivial class of graphs satisfying the conditions of
Proposition~\ref{prop:spoon} consists of graphs avoiding the
1-spoon. By specializing to $k=1$  the proof of
that proposition, we find $C_1=\Cyc(z)$ and $C_2=0$
(since no  core having several cycles avoids the 2-path). Hence 
$$
C(z)= T(z)-\frac {T(z)^2} 2 + \frac 1 2 \left (\log \frac 1 {1-z}
-z -\frac{z^2}2\right).
$$
More generally, consider the class $\cA^{(k)}$ of graphs avoiding the
$k$-spoon, but also the diamond  and the bowtie (both shown in Figure~\ref{fig:ex}):
excluding the latter two graphs 
means  that no graph of
$\cC$ can have several cycles, so that $C_2=0$. Hence the proof of
Proposition~\ref{prop:spoon} immediately gives the following result.

\begin{figure}[b!]
\includegraphics[angle=90,height=8cm]{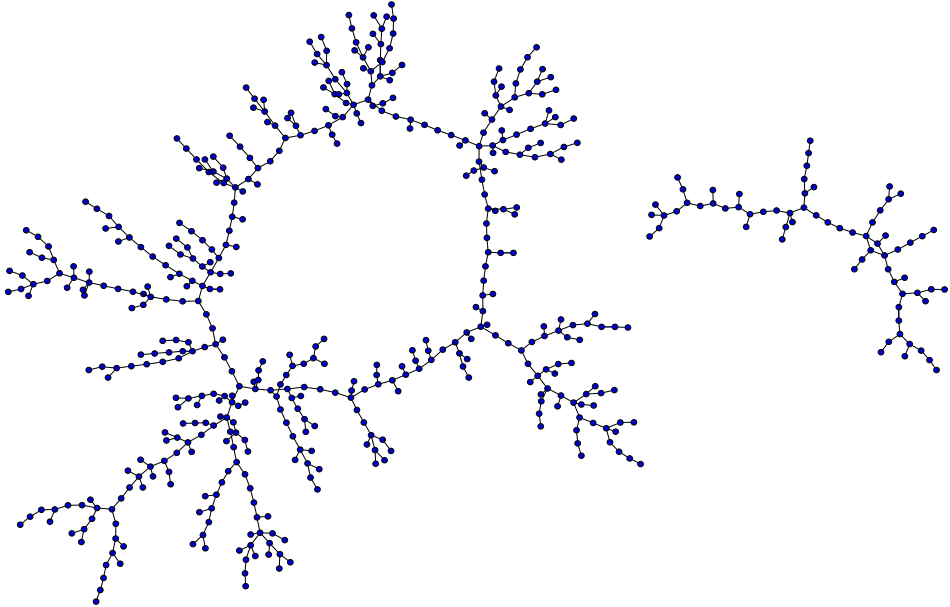}
\caption{A random graph  of size $n=541$ avoiding the diamond, the
  bowtie and the 20-spoon.}
  \label{fig:20spoon}
\end{figure}

\begin{Proposition}[{\bf No diamond,   bowtie or $k$-spoon}]
\label{prop:b+d+s}
Let $k\ge 1$. Let $T(z)$ be the \gf\ of rooted trees, given by $T(z)=ze^{T(z)}$, and
let $T_k(z)$ be the \gf\ of rooted trees of height less than $k$,
given in Example~\ref{ex:forests}.

Let $\cA^{(k)}$ be the class of graphs avoiding the diamond, the bowtie
and the $k$-spoon. The \gf\ of connected graphs of $\cA^{(k)}$ is
$$
C^{(k)}(z)= T(z)-\frac {T(z)^2} 2 + D^{(k)}(z)$$
where
$$
D^{(k)}(z)= \frac 1 2 \left (\log \frac 1 {1-T_{k}(z)}
-T_{k}(z) -\frac{T_{k}(z)^2}2\right).
$$
The class $\cA^{(k)}$ is dominated by trees, and the results of
Propositions~{\rm\ref{prop:trees-dom}},~{\rm\ref{prop:number-trees-dom}} and
{\rm\ref{prop:size-trees-dom}} hold.
In particular,  the probability that a random graph of $\cA_{n}^{(k)}$ is
connected tends to $e^{-C^{(k)}(1/e)}$ as $n\rightarrow \infty$. Since
$
T_{k}(1/e)$ tends to $T(1/e)=1$  as $k$ increases,  this limit
probability tends to $0$.
 \end{Proposition}

A random graph of $\cA_n^{(k)}$ is shown in Figure~\ref{fig:20spoon}
for $k=20$ and $n=541$.
We have also determined the \gf\ of graphs that avoid the 2-spoon.

\begin{Proposition}[{\bf No 2-spoon}]
\label{prop:2-spoon}
Let $T(z)$ be the \gf\ of rooted trees, given by $T(z)=ze^{T(z)}$.
  The \gf\ of connected graphs avoiding the $2$-spoon is
$$
C(z)= T(z)-\frac {T(z)^2} 2 + D(z)
$$
where
$$
  D(z)= \frac{1}{2} \left(\log \frac 1 {1-z  e^{z}}-z  e^{z}
-\frac{z^2  e^{2z}}{2}\right)+ \frac{z^4}{4!}+
z^2  e^{2z} \left(e^z-1-z-\frac {z^2}4\right)
.
$$
The class of graphs avoiding the $2$-spoon is dominated by trees, and the results of
Propositions~{\rm\ref{prop:trees-dom}},~{\rm\ref{prop:number-trees-dom}} and
{\rm\ref{prop:size-trees-dom}} apply.
\end{Proposition}
\begin{proof} We first follow the proof of Proposition~\ref{prop:spoon}: we
  write $C=T-T^2/2+C_1+C_2$, where $C_1$ is given by~\eqref{Cyc-Tk}
  with $T_k=T_2=ze^z$, and $C_2$ counts connected graphs having several
  cycles and avoiding the $2$-spoon. Note that $C_1$ is
the first term in the above expression of $D(z)$. Let us now focus on $C_2$.

In Section~\ref{sec:bowtie} below, we study the class of graphs that avoid the bowtie,
and in particular describe the cores of this class (Proposition~\ref{prop:cores-b}).
Since the bowtie  contains the 2-spoon
as a minor,
 graphs that avoid the 2-spoon
avoid the bowtie as well. Hence we will first determine which cores
of Proposition~\ref{prop:cores-b} have several cycles and avoid the 2-spoon,  and then check
which of their vertices can be replaced  by a  \emm small, tree (that
is, a tree of height 1) without creating a 2-spoon.

Clearly, the cores  of
Proposition~\ref{prop:cores-b} that have  several cycles are those of
Figures~\ref{fig:bt2}, 
\ref{fig:1chord} and~\ref{fig:0chord}.  Among
the cores of Figure~\ref{fig:bt2}, only $K_4$ avoids the
2-spoon. Moreover, none of its vertices can be replaced by a
non-trivial tree. This gives the term $z^4/4!$ in $D(z)$. Among the cores of
Figures~\ref{fig:1chord} and~\ref{fig:0chord}, 
only the ones drawn on the left-hand sides
avoid the 2-spoon. In these cores, only the two vertices of degree at
least 3 can be replaced by a small tree. The resulting graphs are shown in
Figure~\ref{fig:2spoon}  and give together the contribution
$$
\frac{1}2 (z  e^{z})^2 \left(e^z-1-z\right)
+
\frac{1}2 (z  e^{z})^2 \left(e^z-1-z-\frac {z^2}2\right)
$$
(again an application of the symbolic method of~\cite[Chap.~II]{flajolet-sedgewick}). The proposition follows.
\end{proof}
\begin{figure}[htb]
\scalebox{0.7}{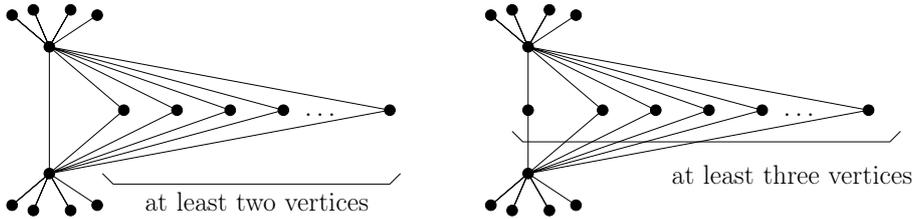}
\caption{Graphs with several cycles  avoiding the $2$-spoon.}
\label{fig:2spoon}
\end{figure}

\section{Excluding the diamond and the bowtie: a logarithmic
  singularity}
\label{sec:b+d}
\begin{figure}[t!]
\includegraphics[angle=90,height=8cm]{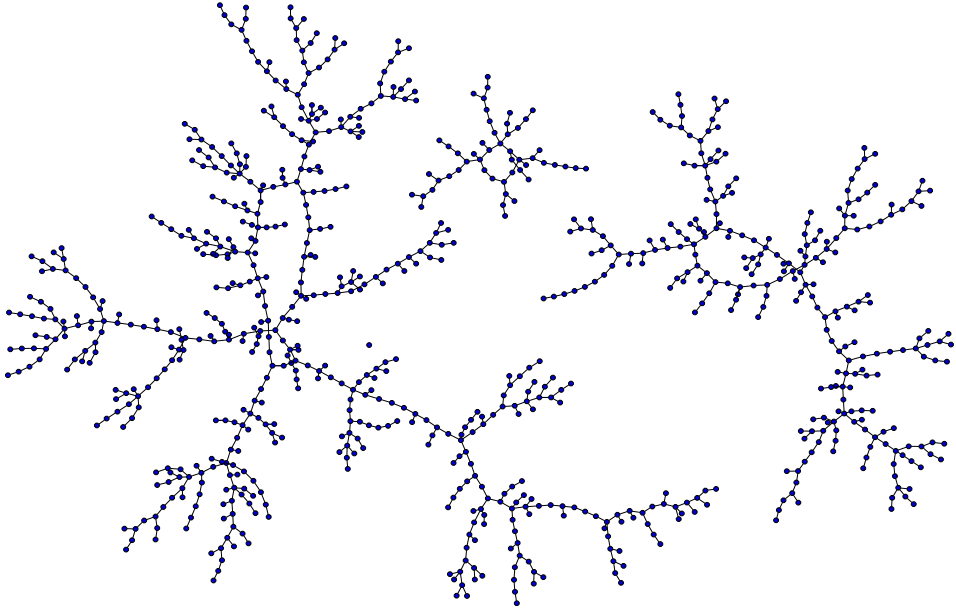}
\caption{A random graph of size $n=859$ avoiding the diamond and the bowtie.}
  \label{fig:random-d+b}
\end{figure}

Let $\cA$ be the class of graphs avoiding the diamond and the bowtie
(both shown in Figure~\ref{fig:ex}).  These are the graphs whose components
have at most one cycle (Figure~\ref{fig:random-d+b}).
 They were studied a long time ago by
R{\'e}nyi~\cite{renyi} and Wright~\cite{wright}, and the following result has now become a routine
exercise.
\begin{Proposition}[{\bf The number of graphs avoiding a diamond and a
  bowtie}]
\label{prop:enum-b+d}
Let $T(z)$ be the \gf\ of rooted trees, defined by $T(z)=ze^{T(z)}$.
  The \gf\ of connected graphs of   $\cA$ is
$$
C(z)= \frac {T(z)} 2 - \frac {3T(z)^2}4 +\frac 1 2 \log \frac 1 {1-T(z)}.
$$
The \gf\ of graphs of $\cA$ is
$
A(z)=e^{C(z)}.
$
  As $n\rightarrow \infty$,
\beq\label{asympt-b+d}
c_n \sim n! \frac {e^n}{4n}\quad \hbox{and} \quad  a_n \sim n!\,
 \frac 1{(2e)^{1/4} \Gamma(1/4)} \frac{e^n}{n^{3/4}}.
\eeq
In particular, the probability that $\cG_n$ is
connected tends to $0$ at speed $n^{-1/4}$ as $n\rightarrow \infty$.
 \end{Proposition}
 \begin{proof}
 The expression of $C(z)$ is obtained by taking the limit
 $k\rightarrow \infty$ in Proposition~\ref{prop:b+d+s}.

We now estimate $c_n$ and $a_n$ via \emm singularity
analysis,~\cite[Sect.~VI.4]{flajolet-sedgewick}. Recall from
Example~\ref{ex:forests} that $T(z)$ has a unique dominant
singularity, at $z=1/e$, with a singular expansion~\eqref{Tsing} valid in a $\Delta$-domain.
Thus $1/e$ is also the unique dominant singularity of $C(z)$ and
$A(z)$, and we have, in a $\Delta$-domain,
\beq\label{CA-sing-b+d}
C(z) \sim \frac 1 4 \log \left(\frac 1 {1-ze}\right)
\quad \hbox{and} \quad  A(z) \sim \frac{1}{(2e)^{1/4}
    (1-ez)^{1/4}}.
\eeq
The asymptotic estimates of $c_n$ and $a_n$ follow.
\end{proof}

\begin{Proposition}[{\bf Number of components --- no bowtie nor    diamond}]
The mean and variance of $N_n$ satisfy:
\[
\E(N_n) \sim \frac{\log n}4,\quad  \quad
\Var(N_n) \sim \frac{\log n}4,
\]
 and the random variable 
$\frac{N_n-  \log n/4  }{\sqrt{\log n/4}}$
converges in law to  a standard normal distribution.
\end{Proposition}
\begin{proof}
Using~\eqref{Tsing}, the estimate~\eqref{CA-sing-b+d} can be refined into
\beq\label{CA-sing-b+d++}
C(z)= \frac 1 4 \log \left(\frac 1 {1-ze}\right) + \lambda + O(\sqrt{1-ze}),
\eeq
where $\lambda$ is a constant, and the proposition is a direct
application of~\cite[Prop.~IX.14, p.~670]{flajolet-sedgewick}.
\end{proof}

The number of
connected components is about $1/4 \log n$. However, the size of the root component is found to be of order $n$. More precisely, we have the following
result.
\begin{Proposition}[{\bf Size of the root  component --- no bowtie nor
  diamond}]
\label{prop:size-b+d}
The normalized variable $S_n/n$ converges in
distribution to a beta law
of parameters $\alpha=1, \beta=1/4$, 
with density $(1-x)^{-3/4}/4$ on $[0, 1]$.
In fact, a local limit law  holds: for $x\in (0,1)$ and $k=\lfloor x
n\rfloor$,
$$
n\, \PP(S_n=k)\rightarrow \frac 1 4 (1-x)^{-3/4}.
$$
   The convergence of moments holds as well: for $i\ge 0$, 
$$
\E(S_n^i) \sim \frac{\Gamma(5/4) i!}{\Gamma(i+5/4)} n^{i}.
$$
\end{Proposition}
\begin{proof} 
Recall that the existence of a local limit law implies the existence
of a global one~\cite[Thm.~3.3]{billingsley}. Thus it suffices to prove the local limit
law. But this is easy,  starting from the rightmost expression
in~\eqref{local}, and using~\eqref{asympt-b+d}.

For the moments, let us start from~\eqref{S-moments}. Our first task
is to obtain an 
estimate of $C^{(i+1)}(z)$ near $1/e$.
Combining~\eqref{CA-sing-b+d++} and~\cite[Thm.~VI.8, p.~419]{flajolet-sedgewick}
gives, for $i\ge 1$, 
$$
C^{(i+1)}(z) \sim \frac{i!}4 \left(\frac e{{1-ze}}\right)^{i+1} .
$$
We multiply this by the estimate~\eqref{CA-sing-b+d} of $A(z)$,
 apply singularity analysis, and finally use~\eqref{asympt-b+d}  to obtain the
 asymptotic behaviour of 
the  $i$th  moment of $S_n$.
Since these moments characterize the above beta  distribution, we 
conclude~\cite[Thm.~C.2]{flajolet-sedgewick} that $S_n/n$ converges in
law to this distribution. 
\end{proof}

\begin{figure}[htb]
\vskip -15mm \includegraphics[width=5cm]{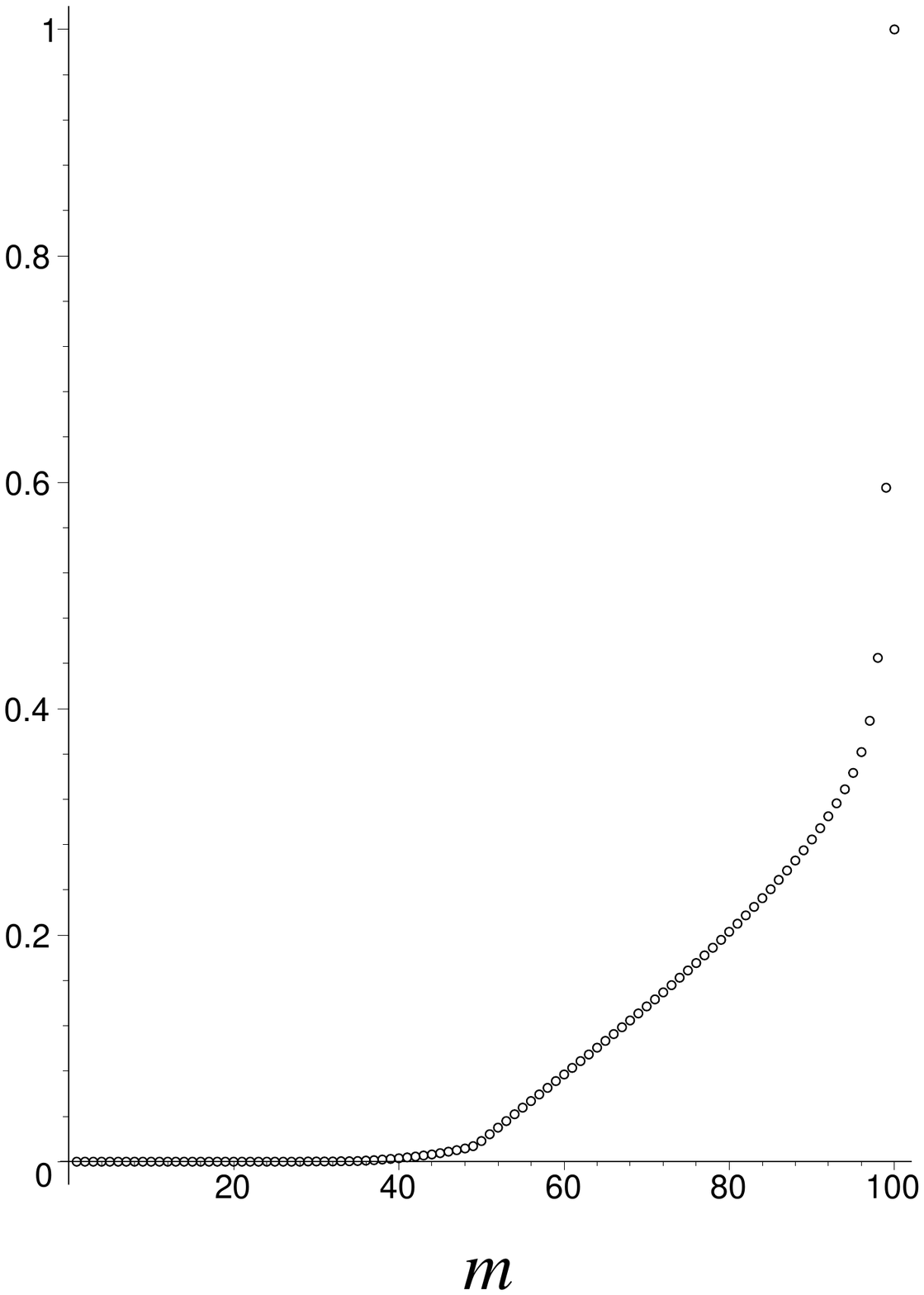}
\hskip -5mm \includegraphics[width=5cm]{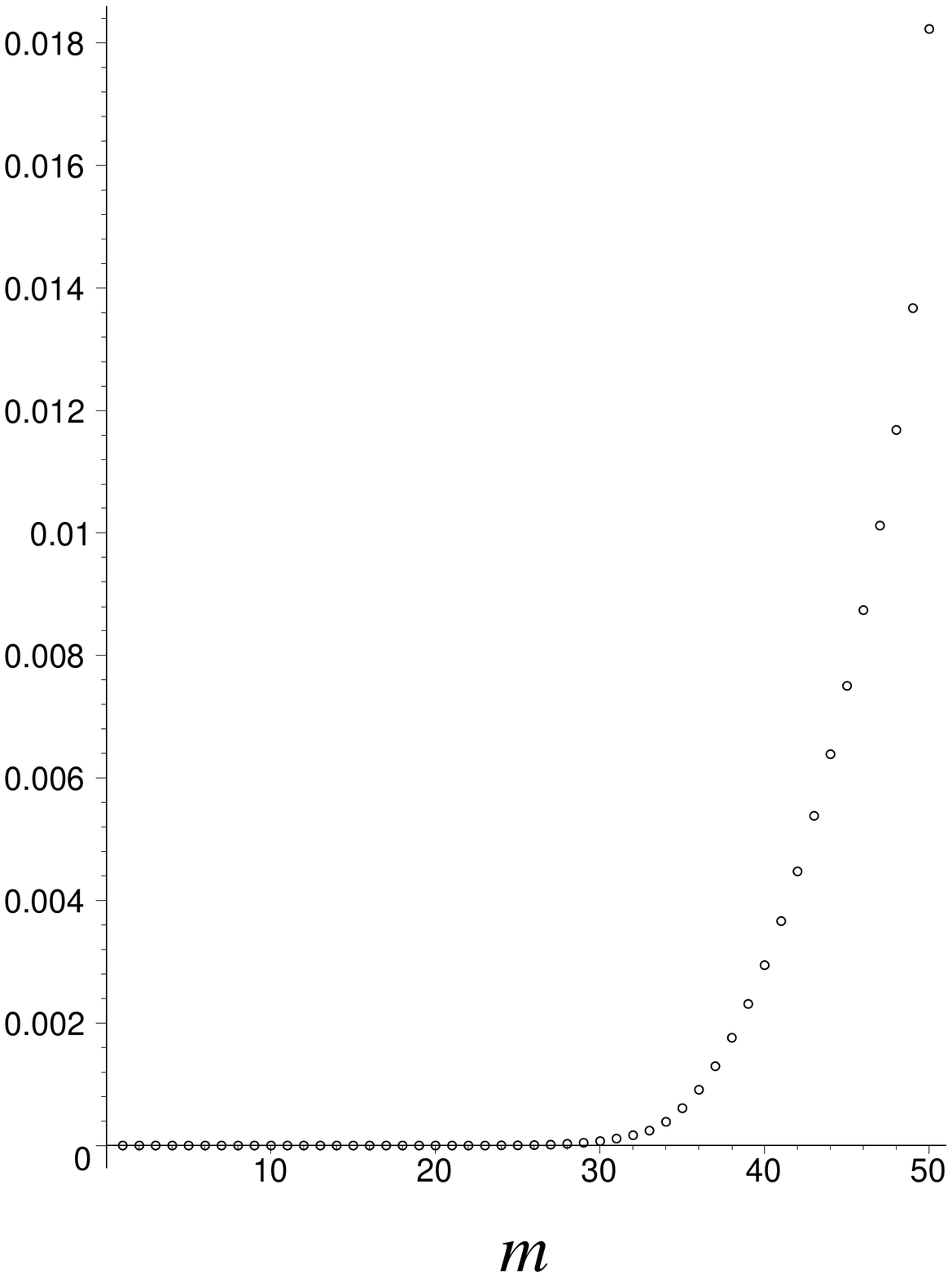}
\hskip -5mm \includegraphics[width=5cm]{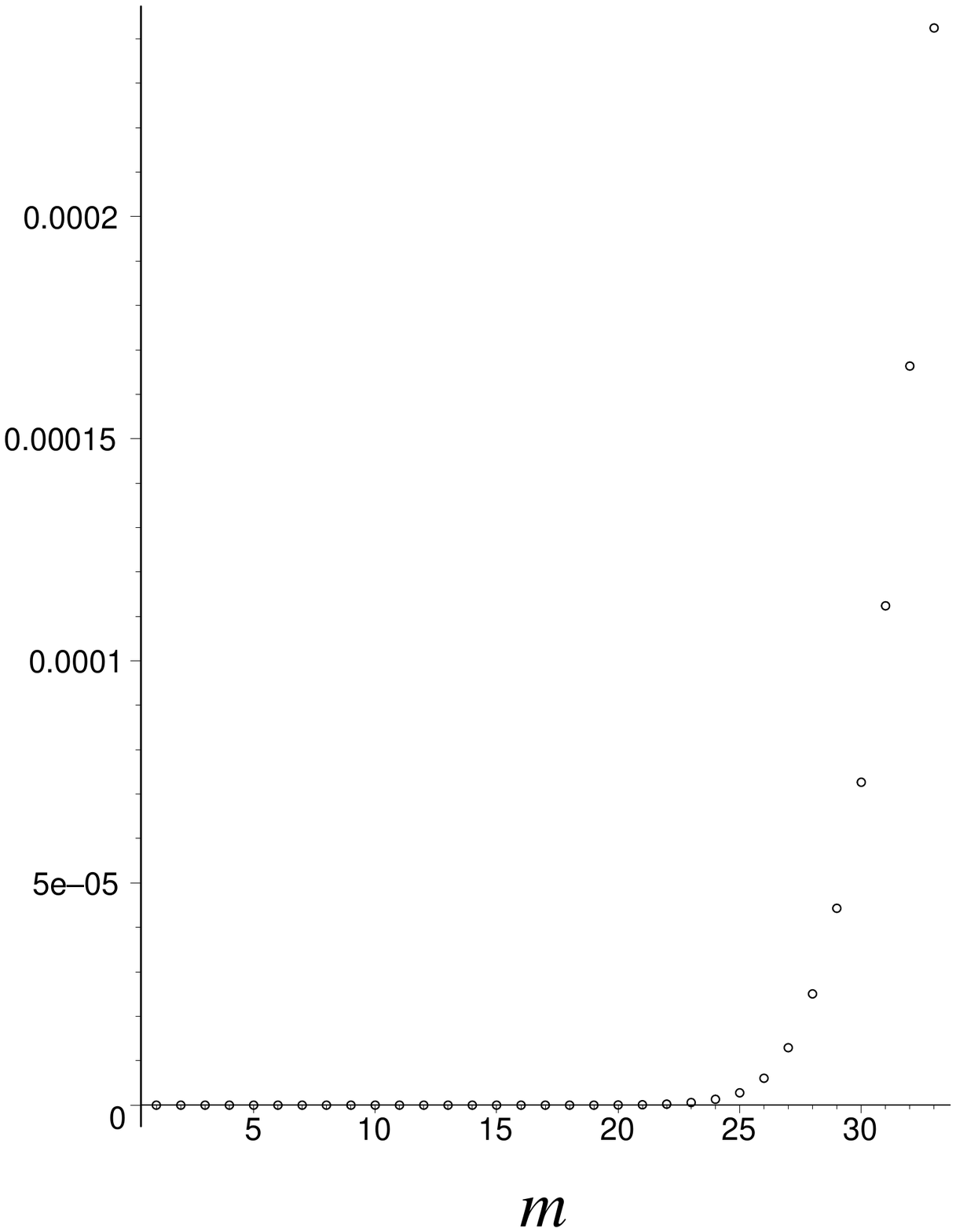}
\vskip -8mm\caption{The distribution function $\PP(L_n <m)$ for
  $n=100$. The first  plot shows the change of regime at $m=n/2$, the
  second the change at  $m=n/3$, the third the change at $m=n/4$.} 
  \label{fig:dickman}
\end{figure}

We conclude with the law of the size of the largest component, which
we derive from general results dealing with components of \emm logarithmic
structures,~\cite{arratia}. The following proposition is illustrated
by Figure~\ref{fig:dickman}.
\begin{Proposition}[{\bf Size of the largest  component --- no bowtie nor
  diamond}]
\label{prop:largest-b+d}
  The normalized variable $L_n/n$ converges in
law to the first component of a Poisson-Dirichlet
distribution of parameter $1/4$: for $x\in (0,1)$,
$$
\PP(L_n <xn) \rightarrow \rho(1/x),
$$
where $\rho: \rs^+ \rightarrow [0,1]$ is the unique continuous
function such that $\rho(x)=1$ for $x\in [0,1]$ and for $x>1$,
$$
x^{1/4} \rho '(x)+\frac 1 4 (x-1)^{-3/4} \rho(x-1)=0.
$$
The function $\rho$ is infinitely differentiable, except at integer
points.

A local limit law also holds: for $x\in (0,1)$ and $1/x \not \in \ns$,
$$
n\, \PP(L_n=\lfloor x n\rfloor)\rightarrow \frac {(1-x)^{-3/4}}{4x} \, \rho\left(\frac{1-x}x\right).
$$
\end{Proposition}
\begin{proof} A decomposable class of graphs $\cA$ is an \emm assembly, in
  the sense of~\cite[Sec.~2.2]{arratia}. In particular, it satisfies
  the conditioning relation~\cite[Eq.~(3.1)]{arratia}:
conditionally to the total size being $n$, the numbers $C_i^{(n)}$
that count connected components of size $i$, for $1\le i \le n$, are
independent.
 When $\cA$ is the
  class of graphs avoiding the diamond and the bowtie,
the estimate~\eqref{asympt-b+d} of $c_n$ tells us
 that this assembly is \emm logarithmic, in
  the sense of~\cite[Eq.~(2.15)]{arratia}: indeed,
  \cite[Eq.~(2.16)]{arratia} holds with $m_i=c_i$, $y=e$ and
  $\theta=1/4$. Our random variable $L_n$ coincides with the random
  variable $L_1^{(n)}$ of~\cite{arratia}. We then apply Theorem~6.12
  and   Theorem~6.8 of~\cite{arratia}: this gives the convergence in
  law of $L_n$ and the local limit law. The distribution function of
  the limit law is given by~\cite[Eq.~(5.29)]{arratia}, and the
  differential equation satisfied by $\rho$ follows
  from~\cite[Eq.~(4.23)]{arratia}.
\end{proof}
\noindent{\bf Remark.} If we push further the singular
expansion~\eqref{CA-sing-b+d++} of $C(z)$, we find a subdominant term in
$\sqrt{1-ze}$, but  its influence is never felt in the
asymptotics results. We would obtain the same  results (with
possibly different constants) for any $C(z)$ having a purely
logarithmic singularity. 

\section{Hayman admissibility and  extensions}
\label{sec:tools}

Our next examples   (Sections~\ref{sec:bounded} to~\ref{sec:bowtie}) deal
with examples where $C(z)$ diverges at $\rho$ with an algebraic
singularity. 
This results in $A(z)$  diverging rapidly at $\rho$.
We then estimate $a_n$ using the saddle point method --- more precisely, with a black
box that applies to \emm Hayman-admissible, (or \emm H-admissible,)
functions. 
Let us first recall what this black box does~\cite[Thm.~VIII.4, p.~565]{flajolet-sedgewick}.
\begin{Theorem}
  \label{thm:saddle}
Let $A(z)$ be a power series with 
real coefficients and
radius  of convergence $\rho \in (0 , \infty]$. Assume that $A(r)$ is positive for $r \in
(R, \rho)$, for some $R \in (0, \rho)$.
Let 
$$
a(r)=r\frac{A'(r)}{A(r)}
\quad \hbox{and} \quad b(r)=
r\frac{A'(r)}{A(r)}+r^2\frac{A''(r)}{A(r)}-r^2\left(\frac{A'(r)}{A(r)}\right)^2.
$$
Assume  that the following three properties hold:
\begin{enumerate}
\item [$\bf H_1$] \rm{[Capture condition]}
$$
\lim_{r\rightarrow \rho} a(r) = \lim_{r\rightarrow \rho} b(r) = +\infty.
$$
\item [$\bf H_2$] \rm{[Locality condition] }
For some function $\theta_0(r)$ defined over $(R,\rho)$,
  and satisfying $0<\theta_0(r)<\pi$, one has, as  $r \rightarrow \rho$,
$$
\sup_{|\theta|\le \theta_0(r)} \left|\frac{A(re^{i\theta})}{ A(r)} e^{-i\theta a(r)+\theta^2b(r)/2} -1\right|\rightarrow 0.
$$
 \item[$\bf H_3$] \rm{[Decay condition] }  As  $r \rightarrow \rho$,
$$
\sup_{|\theta| \in [\theta_0(r), \pi)}\left|\frac{A(re^{i\theta})}{A(r)}{\sqrt{b(r)}}
\right|\rightarrow 0.
$$
\end{enumerate}
We say that $A(z)$ is \emm Hayman admissible.,
Then the $n$th coefficient of $A(z)$  satisfies, as $n\rightarrow
\infty$,
\beq\label{H-est}
[z^n]A(z) \sim \frac{A(\zeta)} {\zeta^n \sqrt{2\pi b(\zeta)}}
\eeq
where $\zeta\equiv \zeta_n$ is the unique solution in $(R, \rho)$ of
the saddle point equation
$
\zeta A'(\zeta)= n A(\zeta).
$
\end{Theorem}
Conditions $\bf H_2$ and $\bf H_3$ are usually stated in terms of \emm
uniform equivalence, as $r\rightarrow \rho$, but we find the above
formulation more explicit. 

The set of H-admissible series has several useful closure
properties~\cite[Thm.~VIII.5, p.~568]{flajolet-sedgewick}. Here is one
that we were not able to find in the literature.
\begin{Theorem}\label{thm:h-h}
Let $A(z)=F(z)G(z)$
where $F(z)$ and $G(z)$ are power series with real coefficients and
radii of convergence 
$0<\rho_{F}<\rho_{G}\le \infty$. Assume that $F(z)$
has non-negative coefficients and  is Hayman-admissible,
 and that $G(\rho_{F})>0$.
Then  $A(z)$ is Hayman-admissible.
\end{Theorem}
\begin{proof}
Let us first prove that the radius of convergence $\rho$ of $A(z)$ is
$\rho_{F}$. Clearly, $\rho\geq \rho_{F}$. 
Now, suppose $\rho>\rho_{F}$. Then $A(z)$ is analytic at $\rho_{F}$. Together with $G(\rho_{F})>0$ this implies that
$F(z)=A(z)/G(z)$ has an  analytic continuation at $\rho_{F}$, which is 
impossible by Pringsheim's Theorem (since $F(z)$ has non-negative
coefficients)~\cite[Thm.~IV.6, p.~240]{flajolet-sedgewick}.  
Note also that $A(r)$ is positive on an interval of
the form $[R, \rho)$ (by continuity of $G$). 
Let us now check the three conditions  of Theorem~\ref{thm:saddle}.
We have
$$
 a(r)=a_{F}(r)+a_G(r),\quad \quad
 b(r)=b_{F}(r)+b_G(r),
$$
with
$$
a_F(r)=r\frac{F'(r)}{F(r)}
\quad \hbox{and} \quad
b_F(r)=r\frac{F'(r)}{F(r)}+r^2\frac{F''(r)}{F(r)}-r^2\left(\frac{F'(r)}{F(r)}\right)^2,$$
and similarly for $a_G$ and $b_G$.

$\bf H_1$ The capture condition  holds for $A$ since it holds
for $F$, given that $G(\rho)>0$ and $\rho_G>\rho$.

$\bf H_2$ 
Choose $ \theta_{0}(r)=\theta_{0}^{F}(r)$
where  $\theta_{0}^{F}(r)$ is a 
function for which $F(z)$ satisfies  $\bf H_2$ and $\bf H_3$. 
We have
\beq\label{p1}
\frac{A(re^{i\theta})}{A(r)}\cdot
e^{-ia(r)\theta+\theta^2 b(r)/2}
=
\frac{F(re^{i\theta})}{F(r)}  e^{-ia_{F}(r)\theta+\theta^2{b_{F}(r)}/2} 
 \frac{G(re^{i\theta})}{G(r)} \cdot
  e^{-ia_{G}(r)\theta+\theta^2{b_{G}^{2}(r)}/{2}} .
\eeq
By assumption, $F$ satisfies the locality condition: hence
\beq\label{p2}
\frac{F(re^{i\theta})}{F(r)} \cdot
e^{-ia_{F}(r)\theta+{b_{F}(r)\theta^2}/{2}} =1+M(r,\theta)
\eeq
where
\beq\label{p3}
\sup_{|\theta |\le \theta_0(r)}\left |M(r,\theta)\right|\rightarrow 0
\eeq
as $r\rightarrow \rho$.
For  $r\in[R,\rho)$ and  $|\theta|\leq \theta_{0}(r)$,
let us expand $\log G(re^{i\theta})$ in powers of $\theta$:
$$
\log G(re^{i\theta})=\log G(r)+i\theta
a_G(r)-\frac{\theta^2}{2}b_G(r)+{\theta^3}S(r,\theta)
$$
where 
$
S(r,\theta)$
is bounded uniformly in a neighborhood of $(\rho, 0)$.
We can assume  that $\theta_0(r) \rightarrow 0$ as $r\rightarrow \rho$
(see~\cite[Eq.~(12.1)]{hayman}).
Thus
\beq\label{p4}
\frac{G(re^{i\theta})}{G(r)} \cdot
e^{-ia_{G}(r)\theta+\theta^2{b_{G}(r)}/{2}} =e^{\theta^3
  S(r,\theta)}
=1+N(r,\theta)
\eeq
where
\beq\label{p5}
\sup_{|\theta |\le \theta_0(r)}\left |N(r,\theta)\right|\rightarrow 0
\eeq
 as $ r\rightarrow \rho$.
Putting together Eqs.~\eqref{p1} to~\eqref{p5}, we obtain that $A(z)$ satisfies
$\bf H_2$.

$\bf H_3$ We have:
\begin{eqnarray*}
 \left| \frac{A(re^{i\theta})}{A(r)}\sqrt{b(r)}\right|& =&
 \left|\frac{F(re^{i\theta}) G(re^{i\theta})}{F(r) G(r)} \sqrt{b_{F}(r)+b_G(r)}\right|\\
&\leq &\left|\frac{F(re^{i\theta})}{F(r)}\sqrt{2  b_{F}(r)} \cdot
  \frac{G(re^{i\theta})}{G(r)}\right| \hskip 20mm \hbox{for } r \hbox{
  close to }\rho,
\end{eqnarray*}
because $b_{F}(r)\rightarrow \infty$ as $r\rightarrow \rho$  while
$b_G(r)$ is bounded around $\rho$.
Also, since  $G$ has radius larger than $\rho$ and $G(\rho)>0$, the term
$G(re^{i\theta})/G(r)$ is uniformly bounded in a neighborhood of the
circle of radius $\rho$.
Since by assumption,  $F(z)$ satisfies $\bf H_3$, this shows that
$A(z)$ satisfies it as well. 
\end{proof}

We will also need a \emm uniform, version of Hayman-admissibility for
series of the form $e^{uC(z)}$. 
\begin{Theorem}
\label{thm:ext-adm}
Let $C(z)$ be a power series with non-negative coefficients and radius
of convergence $\rho$. Assume that $A(z)=e^{C(z)}$ has radius $\rho$ and
is Hayman-admissible. Define
$$
b(r)=r C'(r)+r^{2} C''(r) 
\quad \hbox{and} \quad  V(r)= C(r)-\frac{(r C'(r))^{2}}{r
      C'(r)+r^{2} C''(r)}.
$$
Assume that, as $r\rightarrow \rho$,
\begin{eqnarray}
V(r) &\rightarrow& +\infty,  \label{ext-adm0}\\
\frac{C(r)}{V(r)^{3/2}} &\rightarrow & 0, \label{ext-adm1}\\
b(r)^{1/\sqrt{V(r)}} &=&O(1).\label{ext-adm2}
\end{eqnarray}
Then $A(z,u):=e^{uC(z)}$  
satisfies Conditions {\rm{(1)--(6), (8)}} and {\rm(9)}
of~\cite[Def.~1]{drmota}. If $N_n$ is a sequence of random variables such that
$$
\PP(N_n=i )=\frac{[z^n] C(z)^i}{i![z^n]e^{C(z)}},
$$
then the mean and variance of $N_n$ satisfy:
\beq\label{M1M2}
\E(N_n) \sim  C(\zeta_n) ,\quad  \quad
\Var(N_n) \sim V(\zeta_n),
\eeq
where $ \zeta_n\equiv \zeta$ is the unique solution in $(0, \rho)$ of
the saddle point equation
$
\zeta C'(\zeta)= n$.
Moreover, the normalized version of  $N_n$ converges in law to a
standard normal distribution:
$$
\frac{N_n -\E(N_n)}{\sqrt{\Var(N_n)}} \rightarrow \cN(0,1).
$$
\end{Theorem}
\noindent{\bf Remark.} The set of  series  covered by this theorem
seem to have only a small intersection with the set of series (of the
form $g(z) F(u f(z))$) covered by Section 4 of~\cite{drmota-soria}.
%
\begin{proof}
With the  notation of~\cite[Def.~1]{drmota}, we have
  \beq\label{eps-def}
\begin{array}{lll}
    a(r,u)=c(r,u)=ruC'(r)=ua(r),& &b(r,u)= ruC'(r)+r^2u C''(r)=ub(r),\\
\\
\bar a(r,u)=\bar b(r,u)=uC(r),
&&\varepsilon(r)= \displaystyle \frac{K}{\sqrt{ V(r)}}
\end{array}
\eeq
for a fixed constant $K$. 
Condition (1) of~\cite[Def.~1]{drmota} holds for $R=\rho$, any
$\zeta>0$ and any $R_0\in[0, \rho)$: indeed, the series $A(z,u)$ is
analytic  for $|z|<\rho$ and $u\in \cs$, and 
$A(z,1)$ is positive on $[0, \rho)$.
Conditions (8) and (9) are nothing but our assumptions~\eqref{ext-adm0}
and~\eqref{ext-adm1}. Condition (4) is that $b(r)\rightarrow +\infty$
as $r\rightarrow \rho$: this holds because $A$ is Hayman admissible. 
Condition (5) requires that $b(r,u)\sim b(r,1)$ for $r\rightarrow
\rho$, uniformly for $u \in
[1-\eps(r), 1+\eps(r)]$: this holds because  $b(r,u)/b(r,1)=u$ and $\varepsilon(r) \rightarrow 0$ as
$r\rightarrow \rho$. Condition (6) requires that
$a(r,u)=a(r,1)+c(r,1)(u-1)+O(c(r,1) (u-1)^2))$ uniformly for $r\in (0, \rho)$
and $u\in [1-\eps(r),1+\eps(r)]$. Since
$a(r,u)=a(r,1)+c(r,1)(u-1)$,  this condition obviously holds.

We are thus left with Conditions (2) and (3), which are uniform
versions (in $u$) of the locality and decay conditions  $\bf H_2$ and
$\bf H_3$ defining Hayman admissibility. They can be stated as follows.
\begin{enumerate}
\item [$\bf H_2'$] \rm{[Uniform locality condition] }
There exists $R\in(0, \rho)$ such that   for any $K>0$, there exists a function $\delta(r)$ defined over $(R,\rho)$,
  and satisfying $0<\delta(r)<\pi$, such that, as  $r \rightarrow \rho$,
$$
\sup_{\large{|\theta|\le \delta(r), \atop   |u-1| \le \eps(r)}} \left|\frac{A(re^{i\theta},u)}{ A(r,u)} e^{-i\theta a(r,u)+\theta^2b(r,u)/2} -1\right|\rightarrow 0.
$$
 \item[$\bf H_3'$] \rm{[Uniform decay condition] } As  $r \rightarrow \rho$,
$$
\sup_{|\theta| \in [\delta(r), \pi)  \atop   |u-1| \le \eps(r)}\left|\frac{A(re^{i\theta},u)}{A(r,u)}{\sqrt{b(r,u)}}
\right|\rightarrow 0.
$$
\end{enumerate}
We begin with  $\bf H_2'$.
Since $A(z)$ is H-admissible, let $\theta_0(r)$ be a function for
which $\bf H_2$ (and $\bf H_3$) holds:
$$
\frac{A(re^{i\theta})} {A(r)} e^{-i\theta a(r) +\theta^2
  b(r)/2}=1+M(r,\theta)
$$
where
$$
M(r):=\sup_{|\theta|\le \theta_0(r)} |M(r,\theta)|
$$
tends to 0 as $ r \rightarrow \rho$.  Then, for $u \in
[1-\varepsilon(r), 1+\varepsilon(r)]$,
\begin{eqnarray*}
  \frac{A(re^{i\theta},u)} {A(r,u)} e^{-i\theta a(r,u) +\theta^2
  b(r,u)/2}
&= &\exp\left(u \left(C(re^{i\theta})-C(r) -i\theta a(r) +\theta^2
  b(r)/2 \right) \right)\\
&= &(1+M(r,\theta))^u 
\end{eqnarray*}
where we have taken the principal determination of $\log$ to define
$(1+M(r,\theta))^u = \exp(u \log(1+M(r,\theta)))$   
(because $M(r, \theta)$ is close to $0$).
Thus 
\begin{eqnarray*}
  \sup_{|\theta|\le \theta_0(r)  \atop |u-1| \le \varepsilon(r)}
\left|  \frac{A(re^{i\theta},u)} {A(r,u)} e^{-i\theta a(r,u) +\theta^2
  b(r,u)/2} -1
\right|
&=& \sup_{|\theta|\le \theta_0(r) \atop |u-1| \le \varepsilon(r)}
\left| (1+M(r,\theta))^u -1\right|\\
&\le&
(1+\varepsilon(r)) M(r) + O\left((1+\varepsilon(r)) M(r)^2\right),
\end{eqnarray*}
and this upper bound tends to 0 as $r\rightarrow \rho$. This proves $\bf H_2'$
 with $\delta(r)= \theta_0(r)$.

We finally address  $\bf H_3'$.  Since $A(z)$ satisfies the decay
condition $\bf H_3$, the quantity 
$$
N(r, \theta) :=  
 \frac{A(re^{i\theta})}{A(r)} \sqrt{b(r)}\
$$
satisfies 
\beq\label{c1}
\sup_{|\theta|\in[ \theta_0(r), \pi] } \left| N(r, \theta)\right| \rightarrow 0
\eeq
as $r\rightarrow \rho$.
 We have, for $u \in [1-\varepsilon(r), 1+\varepsilon(r)]$,
\begin{eqnarray*}
 \left|  \frac{A(re^{i\theta},u)}{A(r,u)} \sqrt{b(r,u)}\right|
&=& \left| N(r, \theta)\right| ^u{\sqrt{b(r)}^{1-u}}{\sqrt u} \\
&\le &  \left| N(r, \theta)\right| ^{1+\varepsilon(r)}
{\sqrt{b(r)}^{\varepsilon(r)}}{\sqrt {1+\varepsilon(r)} }
,
\end{eqnarray*}
and this tends to $0$ uniformly for $|\theta| \in [ \theta_0(r), \pi]$
thanks to~\eqref{c1},~\eqref{eps-def},~\eqref{ext-adm0}
and~\eqref{ext-adm2}.

As explained in~\cite{drmota} just below Theorem~2, these eight
conditions give the  estimates~\eqref{M1M2} of $\E(N_n)$ and $\Var(N_n)$ and imply the existence of a gaussian limit law. 
\end{proof}

We finish this section with a simple but useful result of products of series~\cite[Thm.~2]{bender-survey}.
\begin{Proposition}\label{prop:schur}
  Let $F(z) =\sum_n f_n {z^n}$ and $G(z) =\sum_n g_n
  {z^n}$ be power series with radii of convergence
  $0\le \rho_F <\rho_G \le \infty$, respectively. Suppose that
  $G(\rho_F)\not = 0$ and the sequence $(f_n)_{n\ge 0}$ is smooth.
Then   $[z^n]F(z)G(z) \sim G(\rho_F) f_n$.
\end{Proposition}

\section{Graphs with bounded components:  $\boldsymbol {C(z)}$ is a polynomial}
\label{sec:bounded}

Let $\cC$ be a finite class of connected graphs, and let $\cA$ be the
class of graphs with connected components in $\cC$. Note that
$\cA$ is  minor-closed if and only if $\cC$ itself is
minor-closed. This is the case for instance if $\cC$ is the class of
graphs of size at most $k$. In general, we denote by $k$ the size of
the largest graphs of $\cC$. We assume that $C(z)$ is aperiodic.

We begin with the enumeration of the graphs of $\cA$.
The following proposition is a bit more precise than the standard
result on exponentials of polynomials~\cite[Cor.~VIII.2, p.~568]{flajolet-sedgewick},
since it makes explicit the behaviour of 
the term $b(\zeta)$ occurring in the saddle point estimate~\eqref{H-est}.

\begin{Proposition}[{\bf The number of graphs with small components}]
\label{prop:bounded}
Write the \gf\ of   graphs of $\cC$ as
 \beq\label{C-poly}
C(z)=\sum_{i=0}^{k}\frac{c_{i}}{i!}z^i.
\eeq
The \gf\ of  graphs of $\cA$ is
$
A(z)=e^{C(z)}.
$
  As $n\rightarrow \infty$,
\beq\label{an-bounded}
a_n \sim n!\,\frac{1}{\sqrt{2 \pi kn}}\frac{A(\zeta )}{\zeta ^n}
\eeq
where
$\zeta \equiv \zeta_n$ is defined by $\zeta C^{'}(\zeta )=n$ and
satisfies
\beq\label{col-bounded}
\zeta= \alpha n^{1/k} + \beta + O(n^{-1/k})
\eeq with
 \beq\label{alpha}
\alpha= \left( \frac{(k-1)!}{c_k}\right)^{1/k} \quad \hbox{and } \quad
\beta= - \frac{(k-1) c_{k-1}}{kc_k}.
\eeq
 The probability that $\cG_n$ is
connected is of course zero as soon as $n>k$.
 \end{Proposition}
 \begin{proof}
 The series $A(z)$ is H-admissible (\cite[Thm.~VIII.5,
   p.~568]{flajolet-sedgewick}) and Theorem~\ref{thm:saddle}
 applies. The saddle point equation $\zeta C'(\zeta)=n$ is an
 irreducible bivariate
polynomial in $\zeta$ and $n$,
  of degree $k$ in~$\zeta$. Consider $1/n$ as a small parameter
  $x$. By~\cite[Prop.~6.1.6]{stanley-vol2}, the saddle point $\zeta$ admits
  an expansion   of the form 
\beq\label{zeta-exp}
\zeta=\sum_{i\ge i_0} \alpha_i n^{-i/k},
\eeq
for some integer $i_0$ and complex coefficients $\alpha_i$. Using
Newton's polygon method~\cite[p.~499]{flajolet-sedgewick}, 
one easily finds $i_0=-1$ and the
values~\eqref{alpha} of the first two coefficients.

Since $b(r)=rC'(r)+r^2C''(r)$ has leading term $kc_kr^k/(k-1)!$, the
first order expansion of $b(\zeta)$ reads
$$
b(\zeta)= kn+O(n^{(k-1)/k}),
$$
and the asymptotic behaviour of $a_n$ follows.
\end{proof}

Again, the following proposition is more precise than the statement
found, for instance, in~\cite[Thm.~I]{canfield77}, because our
estimates of $\E(N_n)$ and $\Var (N_n)$ are explicit. Note in
  particular that $\E(N_n) \sim  n/ k$ suggests that most components
  have maximal size $k$.
\begin{Proposition}[{\bf Number of components --- Graphs with small components}]
Assume that the coefficient $c_{k-1}$ in~\eqref{C-poly} is non-zero.
The mean and variance of $N_n$ satisfy:
\[
\E(N_n) \sim \frac n k,\quad  \quad
\Var(N_n) \sim  \frac{c_{k-1}}{k\cdot k!}\alpha^{k-1} n^{(k-1)/k},
\]
where $\alpha$ is given by~\eqref{alpha}, and the random variable
$
\frac{N_n- \E(N_n)   }{\sqrt{\Var (N_n)}}
$
converges in law to  a standard normal distribution. 
\end{Proposition}
\begin{proof}
We apply~\cite[Thm.~I]{canfield77} (we can also apply
Theorem~\ref{thm:ext-adm} if $k>3$). Still denoting the saddle point
by $\zeta\equiv \zeta_n$, we just have to find estimates of
$$
\mu_n= C(\zeta)\quad \hbox{ and  } \quad \sigma_n^2= C(\zeta)-\frac{(\zeta C'(\zeta))^{2}}{\zeta
      C'(\zeta)+\zeta^{2} C''(\zeta)}.
$$
Given~\eqref{col-bounded}, we obtain
$$
\mu_n=\frac n k + \frac{c_{k-1}}{k!}\alpha^{k-1} n^{(k-1)/k} +
O(n^{(k-2)/k}),
$$
$$
\zeta^2 C''(\zeta)=(k-1)n -\frac{c_{k-1}}{(k-2)!}\alpha^{k-1}
n^{(k-1)/k}+O(n^{(k-2)/k}),
$$
and finally
$$
\sigma_n^2=\mu_n- \frac{n^2}{n+\zeta^2 C''(\zeta)}
=\frac{c_{k-1}}{k\cdot k!}\alpha^{k-1} n^{(k-1)/k} + O(n^{(k-2)/k}).
$$
\end{proof}
 Since there are approximately $ n/k$ components, one expects
 the size $S_n$ of the root  component to 
be $k$. This is indeed the case, as illustrated in
Figure~\ref{fig:size3}.

\begin{figure}[b!]
\includegraphics[scale=0.08]{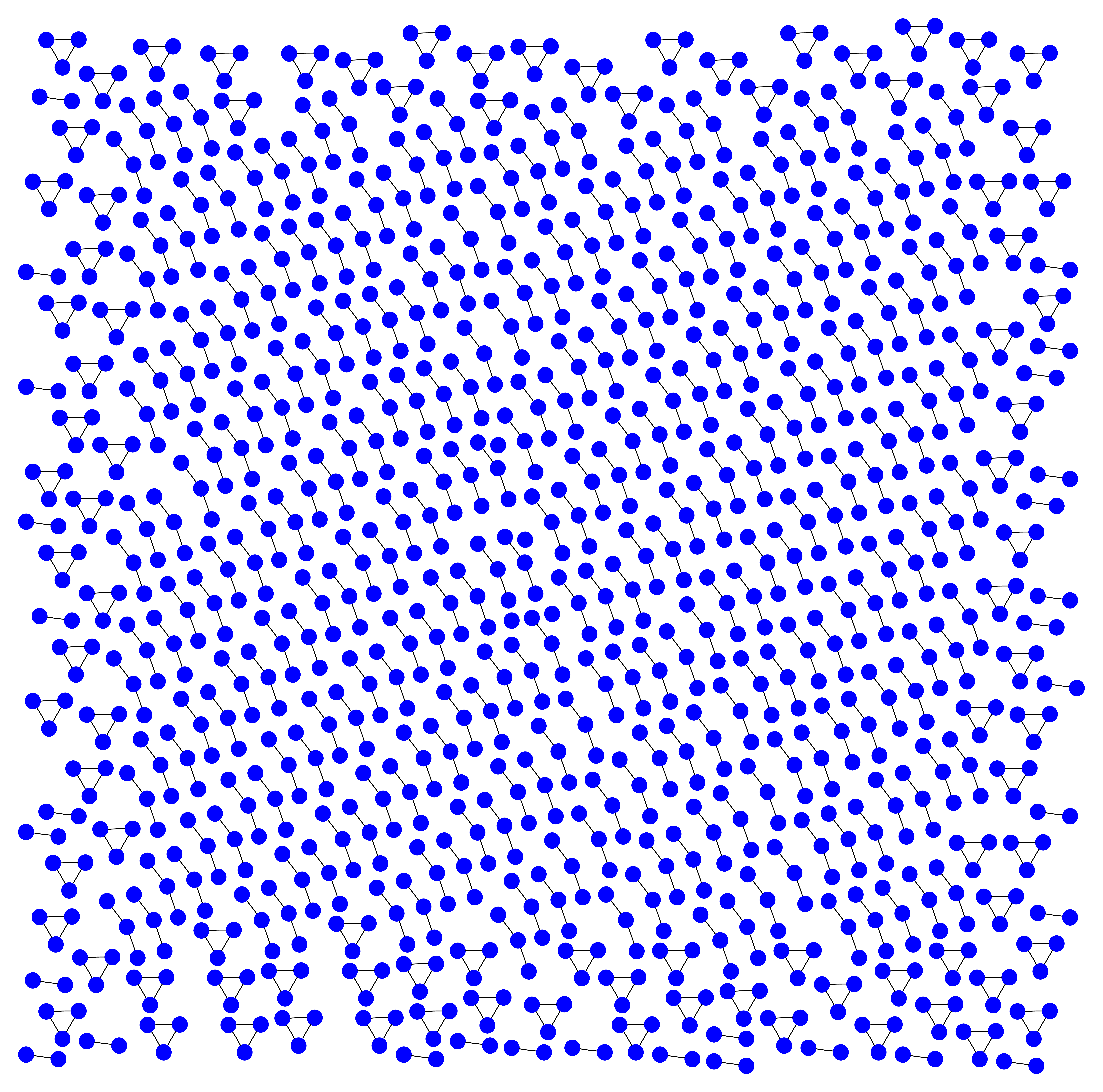}
\caption{A random  graph of size $n=1171$ with component size at most
  3.
Observe that most components have size 3, so that the root component
is very likely to have size 3.} 
  \label{fig:size3}
\end{figure}

\begin{Proposition}[{\bf Size of the   components --- Graphs with
      small components}]\label{prop:???}
 The distribution of $S_n$ converges to a Dirac law at $k$: 
$$
\PP(S_{n}=j)\rightarrow
\begin{cases} 1 & \text{if $j=k$,}
\\
0 &\text{otherwise.}
\end{cases}
$$
The same holds for the size $L_n$ of the largest component.
\end{Proposition}
\begin{proof}
We combine the second formulation in~\eqref{local} with  the
estimate~\eqref{an-bounded} of $a_n$. This gives
$$
\PP(S_n=j)
 \sim
 \frac 1 n  \frac{c_j}{(j-1)! } \frac{A(\zeta_{n-j})}{A(\zeta_n)} \frac
       {\zeta_n^n}{\zeta_{n-j}^{n-j}}.
$$
  Clearly, it suffices to prove that this probability tends to $0$ if
  $j<k$. So let us assume $j<k$. Since $\zeta_n$ is increasing with
  $n$, it suffices to prove that
\beq\label{avoir}
 \frac 1 n  \frac
       {\zeta_n^n}{\zeta_{n-j}^{n-j}}\rightarrow 0.
\eeq
Recall from~(\ref{col-bounded}) and (\ref{zeta-exp}) that $\zeta_n$
admits an expansion of the form
$$
\zeta_{n}=\alpha {n^{1/k}}+\sum_{i=0}^{{k-1}}n^{-i/k}\beta_{i} +
O\left(1/{n}\right)
.$$
This gives, for some constants $\gamma_i$,
$$
n\log \zeta_n = \frac n k \log n +  n \log \alpha +\sum_{i=1}^k\gamma _i
n^{1-i/k} + O(n^{-1/k}).
$$
Hence
$$
(n-j)\log \zeta_{n-j} = \frac {n-j} k \log n +(n-j) \log \alpha +  \sum_{i=1}^k\gamma _i
n^{1-i/k} + O(1).
$$
This gives
$$
n \log \zeta_n -(n-j) \log \zeta_{n-j} -\log n= \frac{j-k}k \log n
+O(1),
$$
and~\eqref{avoir} follows, since $j<k$. Since $L_n \ge S_n$, the
behaviour of $L_n$ is then clear.
\end{proof}

\section{Forests of paths or caterpillars: a simple pole in $\boldsymbol {C(z)}$}
\label{sec:pole}
Let $\cA$ be a decomposable class (for instance defined by
excluding  connected minors), with \gf\ $A(z)=\exp(C(z))$.
Assume that 
\beq\label{C-pole}
C(z)= \frac{\alpha}{1-z/\rho} + D(z),
\eeq
where $D$ has radius of convergence larger than $\rho$. Of course, we
assume $\alpha>0$.
\begin{Proposition}[{\bf The number of graphs --- when $C$ has a simple
  pole}]
\label{prop:pole}
Assume that the above conditions hold, and let $\beta=D(\rho)$.    As
$n\rightarrow \infty$, 
\beq\label{asympt-pole}
c_n \sim n!\, \alpha \rho^{-n} \quad \hbox{and} \quad 
a_n \sim n!\, \frac{\alpha^{1/4} e^{\alpha/2+\beta}}{2\sqrt \pi
  n^{3/4}}\rho^{-n}e^{2\sqrt{\alpha n}}.
\eeq
In particular, the probability that $\cG_n$
is connected tends to $0$ at speed $ n^{3/4} e^{-2\sqrt{\alpha n}}$.
 \end{Proposition}
 \begin{proof}
 The asymptotic behaviour of $c_n$ follows from~\cite[Thm.~IV.10, p.~258]{flajolet-sedgewick}. 
To obtain the asymptotic behaviour of $a_n$, we first write
\beq\label{FG}
A(z)= F(z) G(z) \hskip 5mm
\hbox{ with } \hskip 5mm F(z)=\exp\left({ \frac \alpha{1-z/\rho}} \right)
\hskip 5mm \hbox{ and } \hskip 5mm G(z)=e^{D(z)},
\eeq
where $G(z)$ has radius  of convergence larger than $\rho$. To estimate the coefficients of
$F$, we apply the ready-to-use
results of Macintyre and Wilson~\cite[Eqs.~(10)--(14)]{macintyre}, according to
which, for $\alpha, \gamma>0$ and a non-negative integer $k$,
\beq\label{macintyre}
[z^n]\left( \log \frac 1 {1-z}\right) ^k \frac 1{(1-z)^\gamma} \exp\left(\frac \alpha{1-z}\right) 
\sim\frac {\alpha^{1/4} e^{\alpha
      /2}}{2\sqrt \pi n^{3/4}} \left( \frac n
  \alpha \right) ^ {\gamma /2} \left( \frac{\log n}2\right)^k e^{2\sqrt{\alpha n}}.
\eeq
This gives 
$$
f_n:=[z^n] F(z) \sim \frac{\alpha^{1/4} e^{\alpha/2}}{2\sqrt \pi
  n^{3/4}}\rho^{-n}e^{2\sqrt{\alpha n}}.
$$
This shows in particular  that $f_{n-1}/f_n$ tends to $\rho$
as $n\rightarrow \infty$, so that we can apply
Proposition~\ref{prop:schur} to~\eqref{FG} and  conclude.
\end{proof}

\begin{figure}[b!]
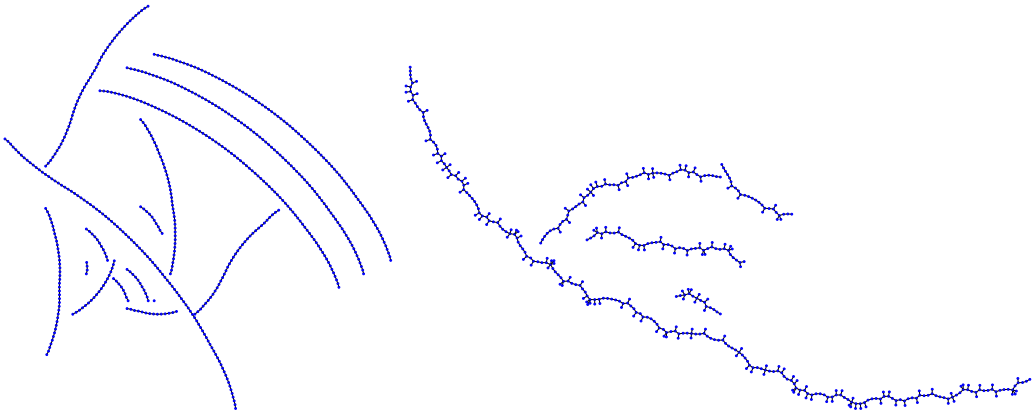

\includegraphics[scale=0.02]{ex-path-forets597-bold}
\includegraphics[scale=0.02,angle=90]{ex-cat-forest486-bold}
\caption{A random forest of paths  of size $n=636$ (left) and a forest
  of caterpillars of size $n=486$ (right).} 
  \label{fig:pole}
\end{figure}

\begin{Proposition}[{\bf Number of components --- when $C$ has a simple
  pole}]
\label{prop:pole-number}
Assume~\eqref{C-pole} holds. The mean and variance of $N_n$ satisfy:
\[
\E(N_n) \sim  \sqrt{  \alpha n},\quad  \quad
\Var(N_n) \sim \sqrt{  \alpha n}/2,
\]
and the random variable
$\frac{N_n-  \sqrt{  \alpha n}   }{ (  \alpha n/4)^{1/4}}$
converges in law to  a standard normal distribution.
\end{Proposition}
\begin{proof}
We   apply Theorem~\ref{thm:ext-adm}. The H-admissibility of $A(z)$  follows from
Theorem~\ref{thm:h-h}, using~\eqref{FG} and the H-admissibility of
$\exp(\alpha/(1-z/\rho))$ (see~\cite[p.~562]{flajolet-sedgewick}). Conditions~\eqref{ext-adm0}--\eqref{ext-adm2}
are then readily checked, using
$$
C(r)\sim \frac{\alpha}{1-r/\rho}, \quad b(r) \sim
\frac{2\alpha}{(1-r/\rho)^3}
\quad \hbox{and} \quad V(r) \sim \frac{\alpha}{2(1-r/\rho)}.
$$
 We thus conclude that the normalized version of
$N_n$ converges in law to a standard normal distribution.
For the asymptotic estimates of $\E(N_n)$ and $\Var(N_n)$, we
use~\eqref{M1M2} with the saddle point estimate
$\zeta_n= \rho- \rho \sqrt{\alpha/n} + O(1/n).$
\end{proof}

 Since there are approximately $\sqrt n$ components, one may expect
 the size $S_n$ of the root  component to 
be of the order of $\sqrt n$. 
%
\begin{Proposition}[{\bf Size of the root  component --- when $C$ has a simple pole}]
\label{prop:pole-root}
 The normalized variable $ S_n/\sqrt{n/\alpha}$ converges in
distribution to a {\rm Gamma}$(2,1)$ law
 of density $x e^{-x}$ on $[0, \infty)$.
In fact, a local limit law  holds: for $x>0$ and $k= \lfloor x \sqrt{n/\alpha}\rfloor$, 
$$
 \sqrt{n/\alpha}\ \PP(S_n=k) \rightarrow x e^{-x}.
$$
The convergence of moments holds as well:  for $i\ge 0$,
$$
\E(S_n^i) \sim(i+1)!  (n/\alpha)^{i/2}.
$$
\end{Proposition}
\begin{proof}
For the local (and hence global) limit law, we simply
combine~\eqref{local} with~\eqref{asympt-pole}.
For the moments, we start from~\eqref{S-moments}, with
$$
C^{(i+1)}(z)= \frac{\alpha (i+1)!}{\rho^{i+1}(1-z/\rho)^{i+2}}+ D^{(i+1)}(z).
$$
Let us first observe that~\eqref{asympt-pole} implies that $a_n/n!$ is
smooth. We can thus apply  Proposition~\ref{prop:schur} to the product
$D^{(i+1)}(z)A(z)$, which gives
$$
\frac{[z^{n-i-1}]D^{(i+1)}(z)A(z)}{n [z^n] A(z)} \sim
\frac{D^{(i+1)}(\rho)}n \frac{a_{n-i-1}}{(n-i-1)!} \frac{n!}{a_n}
\sim \frac{D^{(i+1)}(\rho)}n \rho^i \rightarrow 0.
$$
We thus have
\beq\label{eq13}
\frac{a_n}{(n-1)!}\,\E(S_n^i)\sim  [z^{n-i-1}]
\frac{\alpha (i+1)!}{\rho^{i+1}(1-z/\rho)^{i+2}}\exp \left( \frac
  \alpha{1-z/\rho}+D(z)\right).
\eeq
Now~\eqref{macintyre} gives
\beq\label{eq12}
[z^{n-i-1}]
\frac{\alpha (i+1)!}{\rho^{i+1}(1-z/\rho)^{i+2}}
\exp \left( \frac
  \alpha{1-z/\rho}\right)
\sim \alpha (i+1)!\,\frac{\alpha^{1/4} e^{\alpha/2}}{2\sqrt \pi
  n^{3/4}} \left( \frac n \alpha\right)
^{i/2+1}\rho^{-n}e^{2\sqrt{\alpha n}}.
\eeq
In particular, this sequence of coefficients is smooth.
Hence by Proposition~\ref{prop:schur}, the asymptotic behaviour
of~\eqref{eq13} only differs from~\eqref{eq12} by a factor
$e^{\beta}$, where $\beta=D(\rho)$.  Combined
with~\eqref{asympt-pole}, this gives  the limiting
  $i$th moment of   $S_n$. Since these moments
  characterize the above Gamma  distribution, we can 
conclude~\cite[Thm.~C.2]{flajolet-sedgewick} that
$S_n/\sqrt{n/\alpha}$ converges in law to this distribution. 
\end{proof}

We now present  two classes for which $C(z)$ has
a simple isolated pole (Figure~\ref{fig:pole}): 
forests of paths, and forests of \emm caterpillars, (a
caterpillar is a tree made of a simple path
to which leaves are attached; see Figure~\ref{fig:ex}). In forests of
paths, the excluded minors are the
triangle $K_3$ and the 3-star. The fact that $N_n$ converges in
probability to $\sqrt{n/2}$  for
this class was stated  in~\cite[p.~587]{mcdiarmid}. 
  For  forests of caterpillars, the 
excluded minors are
$K_3$ and the tree shown in Table~\ref{table} ($6$th
line). This class is also considered in~\cite{bernardi}. It is also
the class of graphs of pathwidth 1. 

\begin{Proposition}[{\bf Forests of paths or caterpillars}]
\label{prop:paths}
 The \gfs\ of paths and of caterpillars are respectively
\beq\label{Cp-Cc}
C_p(z)= \frac{z(2-z)}{2(1-z)}
 \quad \quad \hbox{and} \quad  \quad 
C_c(z)= \frac{z^{2} (e^{z}-1)^{2}}{2(1-ze^{z})}+ze^{z}-\frac{z^2}{2}.
\eeq
For both series, Condition~\eqref{C-pole} is satisfied  and Propositions~{\rm\ref{prop:pole},~\ref{prop:pole-number}} and
{\rm\ref{prop:pole-root}} hold. For
paths we have $\rho=1$, $\alpha= 1/2$ and
$\beta:=D(\rho)=0$. For caterpillars,  $\rho\simeq 0.567$ is the only real
such that $\rho e^\rho=1$,
\beq\label{alpha-beta-cater}
\alpha=\frac{(1-\rho)^2}{2(1+\rho)}\simeq 0.06 \quad \hbox{and} \quad
\beta= {\frac {\rho\, \left(10+ 3\,\rho-4\,{\rho}^{2}-{\rho}^{3} \right) }
{4 \left( 1+\rho \right) ^{2}}}\simeq 0.59.
\eeq
\end{Proposition}
 \begin{proof}
 The expression of $C_p(z)$ is straightforward. One can also write
$$
C_p(z)=\frac 1{2(1-z)} + \frac{z-1}2,
$$
which gives $D_p(1)=0$. 
Let us now focus on caterpillars. Let us call \emm star, a tree in
which all vertices, except maybe one, 
have degree 1. By a \emm rooted star, we mean a star with a marked vertex
of maximum degree: hence the root has degree 
0 for a star with 1 vertex, 
1 for a star with 2 vertices,
and at least 2 otherwise. Clearly, there are $n$
rooted stars on $n$ labelled vertices, so that their \gf\ is
$$
S^\bullet(z)=\sum_{n\ge 1} \frac{z^n}{(n-1)!}=ze^z.
$$
The \gf\ of (unrooted) stars is
$$
S(z)=S^\bullet(z) -\frac {z^2}2=ze^{z}-\frac {z^2}2
$$
(because all stars have only one rooting, except the star on 2
vertices which has two).
Now a caterpillar is either a star, or is a (non-oriented) chain of
at least two rooted stars, the first and last having at least 2
vertices each. This gives
$$
C(z)= S(z) + \frac{(S^\bullet(z)-z)^2}{2(1-S^\bullet(z))},
$$
which is equivalent to the right-hand side of~\eqref{Cp-Cc}.

 The series $C_c(z)$ is meromorphic on $\cs$, with a unique dominant pole at
 $\rho$, and its behaviour around this point is easily found using a
 local expansion of $ze^z$ at $\rho$:
$$
C_c(z) = \frac  \alpha{1-z/\rho} 
+\beta+O \left(1-z/\rho \right) ,
$$
with $\alpha$ and $\beta$ as in~\eqref{alpha-beta-cater}.
\end{proof}

 For forests of paths, we have also obtained the limit law of  the
 size $L_n$ of the largest 
component. It is significantly larger than the root component
($\sqrt{n} \log n$ instead of $\sqrt{n}$). 
\begin{Proposition}[{\bf Size of the largest  component --- forests of
    paths}]
In forests of paths,  the (normalized) size of the largest component
converges in law to a Gumbel distribution:
for $x\in \rs$ and as  $n\rightarrow \infty$, 
$$
\PP\left(\frac{L_n-\sqrt{n/2}\log n}{\sqrt {n/2}} < x\right)\rightarrow \exp\left( -\frac{e^{-x/2}}{\sqrt 2}\right).
$$
\end{Proposition}
\begin{proof} We start from~\eqref{df}, where
\beq\label{kn}
k= \sqrt{ n/ 2} \left( \log n +x\right)
\eeq
and  the \gf\ of paths of size less than $k$ is:
\beq\label{Ck-paths}
C^{[k]}(z)= \frac z 2 + \frac{z-z^k}{2(1-z)}.
\eeq
Using a saddle point approach for
integrals~\cite[p.~552]{flajolet-sedgewick}, we will find an estimate of
\beq\label{int}
[z^n] \exp(C^{[k]}(z)) =\frac 1 {2i\pi} \int_{\cC_r} \exp(C^{[k]}(z))
\frac{dz}{z^{n+1}},
\eeq
where the integration contour is any circle $\cC_r$ of center $0$ and  radius
$r<1$. 

Let us first introduce some notation: we denote $C^{[k]}(z)$ by
$K(z)$, the integrand in ~\eqref{int} by $F$, 
and its logarithm by $f$:
$$
K(z)=C^{[k]}(z), \quad F(z)= \frac{\exp(K(z)) }{z^{n+1}}, \quad \quad
f(z)= K(z)- (n+1)  \log z.
$$
We choose the radius $r\equiv r_n$ that satisfies the saddle point equation
$$
F'(r)=0, \quad\hbox{or equivalently}\quad f'(r)=0  \quad\hbox{or} \quad r K'(r) =n+1.
$$
Note that $rK'(r)$ increases from $0$ to $\infty$ as $r$ grows from $0$
to $1$, so that the solution of this equation is unique and simple to
approximate via bootstrapping. We find:
\beq\label{saddle-large}
r=1-\frac 1{\sqrt{2n}}+  \frac{e^{-x/2}}{4\sqrt 2} \frac {\log n}n +
O\!\left(\frac 1 n\right).
\eeq

\smallskip
\noindent{\bf Gaussian approximation.} Let $\theta_0 \in (0, \pi)$. By expanding the function
$g: \theta \mapsto f(re^{i\theta})$ in the neighbourhood of $\theta=0$, we
find, for $|\theta|\le \theta_0$:
\beq\label{log-bound}
\left|f(re^{i\theta})-f(r)+ \theta^2 r ^2 f''(r)/2\right|\le \frac{\theta_0
^3} 6 \sup_{|\alpha| \le \theta_0} \left|g^{(3)}(\alpha)\right|,
\eeq
with 
\begin{eqnarray*}
  \left|g^{(3)}(\alpha)\right|&=&
\left|-ir e^{i\al} f'(r e^{i\al}) -3ir^2 e^{2i\al} f''(r e^{i\al}) -ir^3
e^{3i\al} f'''(r e^{i\al}) \right|
\\
&\le&  K'(r ) + \frac{n+1} r +3 K''(r) +  3 \frac{n+1} {r^2}
 +K'''(r) +  2\frac{n+1} {r^3}.
\end{eqnarray*}
By combining the expression~\eqref{Ck-paths} of $K(z)=C^{[k]}(z)$ and the saddle point estimate~\eqref{saddle-large},
we find that $K'(r)=(n+1)/r \sim n$, that $K''(r) \sim 2\sqrt 2
n^{3/2}$ and finally that $ K'''(r)\sim 12 n^2$. This term dominates the
above bound on $|g^{(3)}(\alpha)|$.
Hence, if 
\beq\label{theta_0-upper}
\theta_0\equiv \theta_0(n) = o(n ^{-2/3}),
\eeq
 we find, by
taking the exponential of~\eqref{log-bound},
$$
F(re^{i\theta})\sim F(r) e^{-\theta^2 r^2f''(r)/2},
$$
uniformly in $|\theta| \le \theta_0$.

\smallskip
\noindent{\bf Completion of the Gaussian integral.} We split the
integral~\eqref{int} into two parts, depending on whether $|\theta|\le
\theta_0$ or $|\theta|\ge\theta_0$. The first part is
$$
\int _{-\theta_0} ^{\theta_0} F(re^{i\theta}) \frac{re^{i\theta}
  d\theta}{2\pi}
\sim \frac {r F(r)}{2\pi} \int _{-\theta_0} ^{\theta_0} e^{-\theta^2 r^2f''(r)/2} e^{i\theta}
  d\theta.
$$
As argued above, $r^2f''(r) \sim K''(r)  \sim 2\sqrt 2
n^{3/2}$. Hence, if we choose $\theta_0\equiv \theta_0(n) $ such that
$\theta_0^2 n^{3/2}\rightarrow \infty$ (which is compatible with~\eqref{theta_0-upper}, for
instance if 
\beq\label{t0-path}
\theta_0=n^{-5/7},
\eeq
which we henceforth assume), we obtain
\begin{eqnarray}
  \label{eq:C0}
\int _{-\theta_0} ^{\theta_0} F(re^{i\theta}) \frac{re^{i\theta}
  d\theta}{2\pi}&\sim& \frac { F(r)}{2\pi \sqrt{f''(r)}}  
\int _{-\theta_0r\sqrt{f''(r)}} ^{\theta_0r\sqrt{f''(r)}} e^{-\alpha^2/2} e^{i\alpha/(r\sqrt{f''(r)})}
  d\alpha \nonumber \\
&\sim&  \frac { F(r)}{2\pi \sqrt{f''(r)}} \int_\rs  e^{-\alpha^2/2}
d\alpha \nonumber \\
& \sim & \frac { F(r)}{\sqrt{2\pi f''(r)}} \nonumber \\
&\sim & \frac{ e^{1/4}e^{\sqrt{2 n}}}{2^{1/4}2\sqrt \pi
  n^{3/4}} \exp\left( -\frac{e^{-x/2}}{\sqrt 2}\right)\label{int-0}
\end{eqnarray}
by~\eqref{saddle-large}.

\smallskip
\noindent{\bf The second part of the integral can be neglected.}
 The second part of the integral~\eqref{int} is
$$
\int _{\theta_0<|\theta|<\pi}  F(re^{i\theta}) \frac{re^{i\theta}
  d\theta}{2\pi},
$$
and we want to prove that it is dominated by~\eqref{int-0}. It
suffices to prove that for $\theta_0<|\theta|<\pi$, 
\beq\label{bound}
|F(re^{i\theta})|= o \left(\frac { F(r)}{\sqrt{ f''(r)}}\right).
\eeq
Let us denote $z= re^{i\theta}$ and $z_0= re^{i\theta_0}$. We have
\begin{eqnarray*}
\frac{|F(re^{i\theta})|}{F(r)}&=& |\exp(K(z)-K(r))|\\
&\le& \exp(|K(z)|-K(r))
=\exp\left(\left| \frac z 2 +\frac{z-z^k}{2(1-z)} \right | 
-\frac r 2 -\frac{r-r^k}{2(1-r)} \right) \\
&\le & \exp\left( \frac{r+r^k}{2|1-z|} 
 -\frac{r-r^k}{2(1-r)} \right)\\
&\le& \exp\left( \frac{r+r^k}{2|1-z_0|} 
 -\frac{r-r^k}{2(1-r)} \right)= \exp\left(-\frac {n^{1/14}}{\sqrt 2}(1+o(1))\right),
\end{eqnarray*}
given the values \eqref{kn}, \eqref{saddle-large} and
\eqref{t0-path} of $k$, $r$ and $\theta_0$. Since $f''(r) \sim 2\sqrt
2 n^{3/2}$, we conclude that~\eqref{bound} holds.

\smallskip
\noindent{\bf Conclusion.} We have now established that the
integral~\eqref{int} is dominated by its first part, and is thus
equivalent to~\eqref{int-0}. To obtain the limiting distribution
function, it remains to divide this estimate by $a_n/n!$.
The asymptotic behaviour of $a_n$
is given by~\eqref{asympt-pole}, with $\alpha=1/2$, $\rho=1$ and
$\beta=0$, and this concludes the proof.
\end{proof}

\section{Graphs with maximum degree 2: a simple pole and a logarithm   in~$\boldsymbol { C(z)}$}   
\label{sec:pole-log}
\begin{figure}[t!]
\includegraphics[scale=0.02,angle=90]{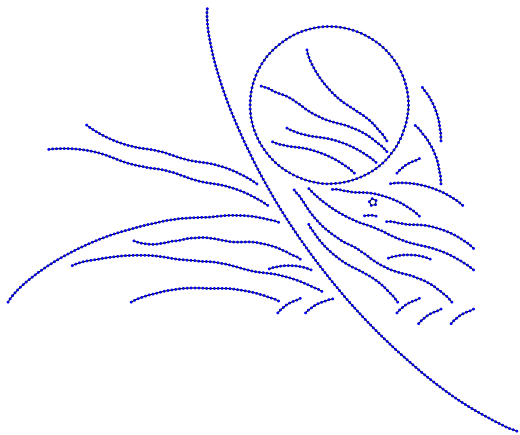}
\caption{A random graph of size $n=1034$ avoiding the 3-star.}
  \label{fig:maxdeg2}
\end{figure}
Let $\cA$ be the class of graphs of maximum degree 2, or equivalently,
the class of graphs avoiding  the 3-star (Figure~\ref{fig:maxdeg2}). The connected components
of such graphs are   paths or cycles. This class differs from those
studied in the previous section in that the series $C(z)$ has now, in
addition to a simple pole, a logarithmic singularity at its radius of
convergence $\rho$. As we shall
see, the logarithm  changes the asymptotic behaviour of the numbers
$a_n$, but the other results remain unaffected. The proofs are very similar to
those of the previous section.

\begin{Proposition}[\bf{The number of graphs of maximum degree 2}]
\label{prop:max2}
  The number of connected graphs (paths or cycles) of size $n$ in the class
  $\cA$ is $c_n=n!/2+(n-1)!/2$ for $n \ge 3$ (with $c_1=c_2=1$)
  and the associated \gf\  is  
$$
C(z)= \frac{z(2-z+z^2)}{4(1-z)}
+ \frac 1 2 \log \frac 1 {1-z}.
$$
The \gf\ of graphs of $\cA$ is
$$
A(z)=e^{C(z)}= \frac 1 {\sqrt{1-z}} \exp\left( \frac{z(2-z+z^2)}{4(1-z)}\right).
$$
  As $n\rightarrow \infty$,
$$
a_n \sim n!\,\frac{1}{2\sqrt {e\pi} n^{1/2}} e^{\sqrt {2n}}.
$$
In particular, the probability that  $\cG_n$ is
connected tends to $0$ at speed $n^{1/2}e^{-\sqrt {2n}}$ as $n\rightarrow \infty$.
 \end{Proposition}
 \begin{proof}
 Again, the exact results are elementary. To obtain the asymptotic behaviour of
 $a_n$, we write 
\beq\label{FG-2}
A(z)=F(z)G(z) \quad \hbox{with} \quad  F(z)= \frac1{\sqrt{1-z}} \exp\left( \frac 1{2(1-z)}\right)
\quad \hbox{and} \quad G(z)= \exp\left(-\frac 1 2 -\frac{z^2}4 \right)
\eeq
and  combine Proposition~\ref{prop:schur}  with
\eqref{macintyre}.
\end{proof}
For the number of components, we find the same behaviour as in the
case of a simple pole (Proposition~\ref{prop:pole-number} with
$\alpha=1/2$). We have also determined the expected number of
cyclic components.
\begin{Proposition}[{\bf Number and nature of components --- Graphs of maximum degree 2}]
The mean and variance of $N_n$ satisfy:
\[
\E(N_n) \sim  \sqrt{n/2} ,\quad  \quad
\Var(N_n) \sim {\sqrt{n/8}} ,
\]
 and the random variable 
$\frac{N_n-  \sqrt{n/2}  }{(n/8)^{1/4}}$
converges in law to  a standard normal distribution.

The expected number of cycles in $\cG_n$ is of order $(\log n)/4$.
\end{Proposition}
\begin{proof}
We want to apply Theorem~\ref{thm:ext-adm}. To prove that $A(z)$ is
Hayman-admissible, we apply
Theorem~\ref{thm:h-h} to~\eqref{FG-2}. This reduces our task to
proving that $F(z)$ is H-admissible, which is done along the same
lines as~\cite[Ex.~VIII.7, p.~562]{flajolet-sedgewick} (see also the
footnote of~\cite[p.~92]{hayman}, 
and Lemma~1 in~\cite{drmota-soria}). 
Conditions~\eqref{ext-adm0}--\eqref{ext-adm2}
are readily checked.
The asymptotic estimates of $\E(N_n)$ and $\Var(N_n)$ are obtained
through~\eqref{M1M2}, using the saddle point estimate $\zeta_n=1
-1/\sqrt{2n} +O(1/n)$. 

\medskip
The bivariate \gf\ of graphs of $\cA$, counted by the size (variable
$z$) and the number of cycles (variable $v$) is 
$$
\tilde A(z,v)=\exp\left(z+ \frac{z^2}{2(1-z)} + v \Cyc(z)\right),
$$
where $\Cyc(z)$ is given by~\eqref{bt1}.
 By differentiating with respect to $v$, the expected number of
cycles in $\cG_n$ is found to be:
$$
\frac{[z^n]\Cyc(z) A(z)}{[z^n] A(z)}.
$$
The asymptotic behaviour of $[z^n] A(z)=a_n/n!$ has been established
in Proposition~\ref{prop:max2}. We determine an estimate of
$[z^n]\Cyc(z) A(z)$ in a similar fashion, using a combination of
Proposition~\ref{prop:schur}  and~\eqref{macintyre}. We find
$$
[z^n]\Cyc(z) A(z)\sim \frac{\log n}{8\sqrt{e\pi}n^{1/2}}e^{\sqrt{2n}},
$$
and the result follows.
\end{proof}
 The size of the root  component is still described by
Proposition~\ref{prop:pole-root}, with $\alpha=1/2$. The proof is very
similar, with now 
$$
C^{(i+1)}(z)= \frac{i!}2 \frac{2+i-z}{(1-z)^{i+2}} - \frac 1 2 {\bf 1}_{i=1},
$$
where ${\bf 1}_{i=1}$ is $1$ if $i=1$ and is 0 otherwise.
\qed

\section{Excluding the bowtie: a singularity in $\boldsymbol{(1-z/\rho)^{-1/2}}$}
\label{sec:bowtie}
We now denote by $\cA$ the class of graphs avoiding the bowtie
(Figure~\ref{fig:bowtie}). 
The following proposition answers a question
raised in~\cite{mcdiarmid-star}. 

\begin{figure}[htb]
\includegraphics[height=8cm]{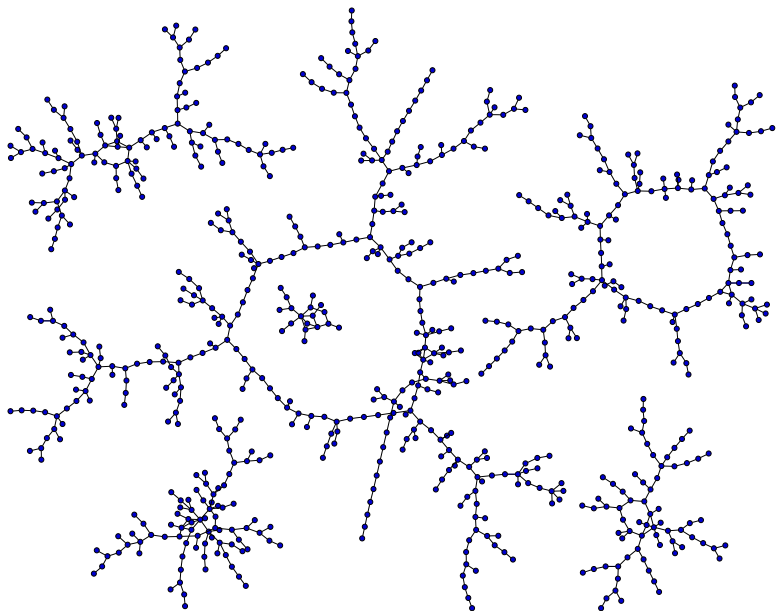}
\caption{A random graph of size $n=758$ avoiding the bowtie.}
  \label{fig:bowtie}
\end{figure}

\begin{Proposition}[\bf{The \gf\  of graphs avoiding a bowtie}]
\label{prop:bowtie}
 Let $T\equiv T(z)$ be the \gf\ of rooted trees, defined by $T(z)=ze^{T(z)}$. The \gf\ of  connected graphs  in the class  $\cA$ is 
\beq\label{C-bowtie}
C(z)= 
{\frac {{T}^{2} \left( 1-T+{T}^{2} \right) {{\rm e}^{T}}}{1-T}}
+\frac 1 2 \log \left(\frac 1{1-T}\right)
+{\frac {T
 \left( 12-54\,T+18\,{T}^{2}-5\,{T}^{3}-{T}^{4} \right) }{24(1-T)}}.
\eeq
The \gf\ of graphs of $\cA$ is
$
A(z)=e^{C(z)}.
$
 \end{Proposition}
This is the most delicate enumeration of the paper. The key point is
the following characterization of cores (graphs of minimum degree 2)
avoiding the bowtie.
\begin{Proposition}
  \label{prop:cores-b} 
The cores that avoid the bowtie are:
  \begin{itemize}
  \item the empty graph,
\item all cycles,
\item $K_4$, with one edge possibly subdivided, as shown in
Figure~{\rm\ref{fig:bt2}},
\item the graphs of Figures~{\rm\ref{fig:1chord}} and~{\rm\ref{fig:0chord}}.
  \end{itemize}
\end{Proposition}
  We will first
 establish a number of properties of cores avoiding a bowtie. Recall that a \emm chord, of a
cycle $C$ is an edge, not in $C$, joining two vertices of $C$. 

\begin{Lemma}
\label{lem:chords}
Let $C=(v_0, v_1, \ldots, v_{n-1})$ be a cycle in a core $G$ avoiding
the bowtie.
Let us write $v_{n}=v_0$ and  $v_{n+1}=v_1$.
  Every chord of $C$ joins vertices that are at distance $2$ on $C$
  (we say that it is a \emm short, chord). Moreover,  $C$ has at
most two chords. If it has two chords, say $\{v_0, v_2\}$ and $\{v_i,
v_{i+2}\}$, with $1\le i \le n-1$ , then
$v_i=v_1$ or $v_{i+2}=v_1$.
\end{Lemma}
\begin{proof}
  If a chord were not short, contracting it (together with some edges
  of $C$) would give a bowtie.  Figure~\ref{fig:2chords} then proves the
  second statement, which can be loosely restated as follows: the two
  chords  cross and their four endpoints are consecutive on $C$.
\end{proof}
\begin{figure}[htb]
\includegraphics[scale=0.7]{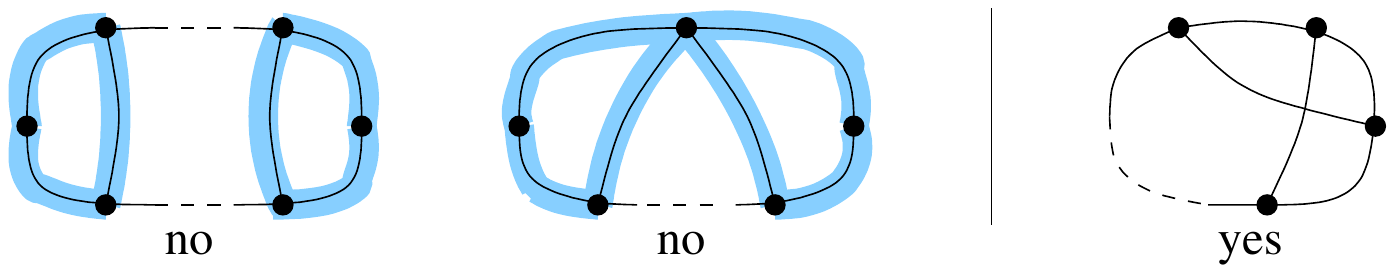}
\caption{The relative positions of two (short) chords in a cycle.
The first two configurations are not valid, as these graphs contain a bowtie
(shown with gray/blue edges).}
  \label{fig:2chords}
\end{figure}

\begin{Lemma}
\label{lem:external}
  Let $C$ be a cycle of maximal length in a core $G$ avoiding the
  bowtie. Let $v$ be an \emm external, vertex, that is, a vertex not belonging to $C$. Then $v$ is
  incident to exactly two edges, both ending on $C$. The endpoints of
  these edges are at distance $2$ on $C$.
\end{Lemma}
\begin{proof}
  Since $G$ is a core, $v$ belongs to a cycle $C'$. Since $G$ is connected
  and avoids the bowtie, $C'$ shares at least two  vertices
  with $C$. Thus let $P_1$ and $P_2$ be two vertex-disjoint paths
  (taken from $C'$) that
  start from $v$ and end on $C$ without  hitting $C$ before. Let
  $v_1$ and $v_2$ be their respective endpoints on $C$. Then
  $v_1$ and $v_2$  lie at distance at least 2 on 
  $C$, otherwise $C$ would not have maximal length. Now contracting
  the path $P_1P_2$ into a single edge gives a chord of $C$. By the
  previous lemma, this chord must be short, so that $v_1$ and $v_2$
  are at distance exactly 2. 
Since $C$ has maximal length,  $P_1$ and $P_2$ have length 1 each, and
thus are edges.
 
 Assume now that $v$ has degree at least 3, and let $e$ be a third
 edge (distinct from $P_1$ and $P_2$) adjacent to $v$. Again, $e$
 must belong to a cycle, sharing at least two vertices with $C$, and
 the same argument as before shows that $e$ ends on $C$. But then
 Figure~\ref{fig:rome} shows that  $G$  contains a bowtie.
\end{proof}

\begin{figure}[htb]
\scalebox{0.7}{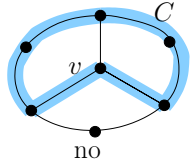}
\caption{A cycle $C$ with an external vertex $v$ of degree at least 3.
The cycle shown in gray/blue has a chord that is not short, the
contraction of which gives a bowtie.}
\label{fig:rome}
\end{figure}

\begin{Lemma}
\label{lem:2chords}
  Let $C$ be a cycle of maximal length in a core $G$ avoiding the
  bowtie. If $C$ has two chords, it contains all vertices  of $G$.
\end{Lemma}
\begin{proof}
  Let $e_1$ and $e_2$ be the two chords of 
   $C$. Lemma~\ref{lem:chords} describes their relative positions.
   Let $v$ be a   vertex not in $C$. Lemma~\ref{lem:external} describes how it is connected to $C$. Contract one of the two edges
  incident to $v$ to obtain  a chord of $C$. By Lemma~\ref{lem:chords}, this
  chord must be a copy of $e_1$ or $e_2$. But then
  Figure~\ref{fig:2chord+ext}(a) shows that $G$   contains a bowtie
  (delete the two bold edges).
\end{proof}
\begin{figure}[htb]
\scalebox{0.7}{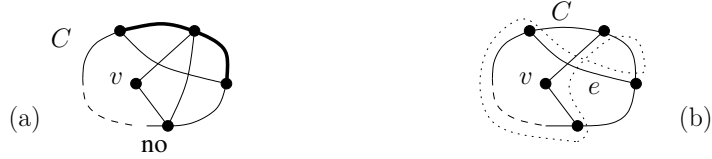}
\caption{A cycle $C$ with (a) two chords and an external vertex $v$ (b) one chord and an external vertex $v$.}
\label{fig:2chord+ext}
\end{figure}

\begin{Lemma}
\label{lem:1chord}
   Let $C$ be a cycle of maximal length in a core $G$ avoiding the
  bowtie. If $C$ has a chord $e$,
all  external vertices of $C$ are adjacent to the endpoints of $e$.
\end{Lemma}
\begin{proof}
  Let $v$ be external to $C$. Contract one of the incident
  edges. This gives a chord $e'$. If $e'$ is a copy of $e$, then we
  are done. Otherwise, the relative positions of $e$ and $e'$ are
  described by Lemma~\ref{lem:chords}. But then Figure~\ref{fig:2chord+ext}(b)
  shows that  $C$ has not maximal length (consider the cycle shown with the
  dotted line).
\end{proof}
\begin{Lemma}
\label{lem:2ext}
    Let $C$ be a cycle of maximal length in a core $G$ avoiding the
  bowtie. If $C$ has several external vertices, they are adjacent to
  the same points of $C$.
\end{Lemma}
\begin{proof}
Consider two external vertices $v_1$ and $v_2$. Lemma~\ref{lem:external} describes
how each of them is connected to $C$.   Contract an edge incident to $v_1$ and
an edge incident to $v_2$. This gives two  chords of $C$. Either these two
chords are copies of one another, which means that $v_1$ and $v_2$ are
adjacent to the same points of $C$. Or the relative position of these
two chords is as described in 
Lemma~\ref{lem:chords}. But then Figure~\ref{fig:2ext}
shows that $G$ contains a bowtie (contract $e$). 
\end{proof}
\begin{figure}[htb]
\scalebox{0.7}{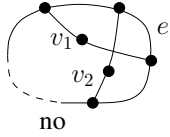}
\caption{A cycle $C$ with  two external vertices.}
\label{fig:2ext}
\end{figure}

\begin{proof}[Proof of Propositions~\ref{prop:bowtie} and~\ref{prop:cores-b}]
 Observe  that a graph $G$ avoids the bowtie if and only if
its core 
(defined as its maximal subgraph of degree 2)
avoids it.  
Hence, if  $\bar C(z)$ denotes the \gf\ of non-empty cores  avoiding the bowtie, we have 
\beq
\label{C-cores}
C(z)=T(z)-\frac{T(z)^2}2 +\bar C (T(z)).
\eeq
Using the above lemmas, we can now describe and count non-empty cores avoiding the
bowtie. We start with cores reduced to a cycle: their contribution 
to $\bar C(z)$ is given by~\eqref{bt1}.
We now consider   cores $G$ having several cycles. Let $C$ be a cycle
of $G$ of maximal length, chosen so that it has the largest possible number of
chords. By Lemma~\ref{lem:chords}, this number is 2, 1 or 0.

If $C$ has two chords, it contains all vertices of $G$
(Lemma~\ref{lem:2chords}).
 By Lemma~\ref{lem:chords} and  Figure~\ref{fig:2chords} (right),  either all
vertices have degree 3 and $G=K_4$, or $G$ consists of $K_4$ 
where one edge is subdivided
(Figure~\ref{fig:bt2}).

\begin{figure}[htb]
\scalebox{0.7}{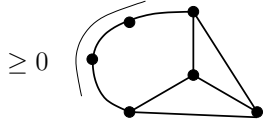}
\caption{$K_4$ with a subdivided edge.}
\label{fig:bt2}
\end{figure}

\noindent This gives the \gf
\beq\label{bt2}
\frac {z^4}{4!} + \frac {z^4}{4!} \cdot 6\cdot \frac z{1-z},
\eeq
where, in the second term, we read first the choice of the 4 vertices
of degree 3 forming a $K_4$, then the choice of one edge of this $K_4$, and
finally the choice of a (directed) path placed along this edge.

Assume now that $C$ has exactly one chord $e$. By Lemma~\ref{lem:1chord},
all external vertices are 
adjacent to the endpoints of $e$. To avoid problems with symmetries, we 
count separately the cores where $C$ has length $4$, or length $\ge 5$
(Figure~\ref{fig:1chord}). This gives the \gf
\beq\label{bt3}
\frac {z^2}2 (e^z-1-z) + \frac{z^2}2(e^z-1)\frac{z^2}{1-z}.
\eeq
In the second term, the factor $z^2/(1-z)$ accounts for the directed
chain of vertices of degree 2 lying on the maximal cycle.

\begin{figure}[htb]
\scalebox{0.7}{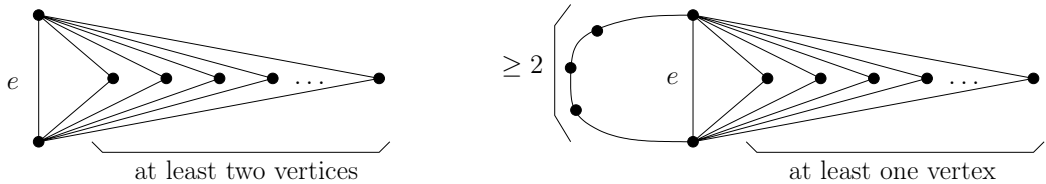}
\caption{A maximal cycle  with one chord $e$ and external vertices.}
\label{fig:1chord}
\end{figure}

Assume finally that $C$ has no chord. By Lemma~\ref{lem:2ext}, all
external vertices are adjacent to the same points of $C$. Again, we
treat separately the cases where $C$ has length $4$, or length $\ge 5$
(Figure~\ref{fig:0chord}). This gives the \gf
\beq\label{bt4}
\frac {z^2}2 (e^z-1-z-z^2/2) + \frac{z^2}2(e^z-1-z)\frac{z^2}{1-z}.
\eeq
\begin{figure}[htb]
\scalebox{0.7}{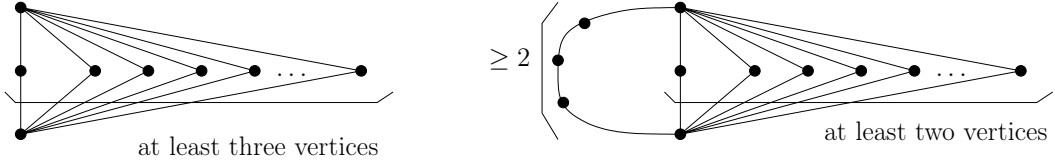}
\caption{A maximal cycle $C$ with no chord and several external vertices.}
\label{fig:0chord}
\end{figure}

Putting together the
contributions~\eqref{bt1},~\eqref{bt2},~\eqref{bt3} and~\eqref{bt4}
gives the value of $\bar C(z)$ (the \gf\ of cores) from which we obtain
the series $C(z)$ using~\eqref{C-cores}.
\end{proof}

We now derive asymptotic results from Proposition~\ref{prop:bowtie}.
\begin{Proposition}[\bf{The asymptotic number of graphs avoiding the
    bowtie}]
\label{prop:b-asympt}
   As $n\rightarrow \infty$,
$$
c_n \sim n! \, \frac{e-5/4}{\sqrt{ 2\pi}} \frac{e^{n}}{\sqrt{ n}}
$$
{and}
$$
a_n \sim n!   \,\frac{(e-5/4)^{1/6}e^{ {19}/8-{11e}/3}
}{\sqrt {6\pi}}
\frac{e^{n}}{{ n}^{2/3}} \,
\exp\left({ \frac 32 (e-5/4)^{2/3} n^{1/3}}\right).
$$
\end{Proposition}
\begin{proof}
  Let us first recall that the series $T(z)$ has radius  of
  convergence $1/e$, and can
  be continued analytically on the domain $\cD:=\cs\setminus[1/e, +\infty)$. In fact,
  $T(z)=-W(-z)$, where $W$ is the (principle branch of the) \emm
  Lambert function,~\cite{MR1414285}. The singular behaviour of $T(z)$
near $1/e$  is given by~\eqref{Tsing}. Moreover, the image of $\cD$ by $T$
avoids the half-line $[1, +\infty)$.

It thus follows from the expression~\eqref{C-bowtie} of $C(z)$  that $C(z)$ and $A(z)$ are analytic in the
domain $\cD$. Moreover, we derive from~\eqref{Tsing} that, as $z$
approaches $1/e$ in a $\Delta$-domain,
\beq\label{C-b-est}
C(z) \sim \frac{e-5/4}{\sqrt 2 \sqrt{1-ze}}.
\eeq
The above estimate of $c_n$ then follows from singularity analysis.

We now embark with the  estimation of $a_n$. We first
prove (see Proposition~\ref{prop:H1} in the  appendix) 
that $A(z)$ is
H-admissible. We then apply Theorem~\ref{thm:saddle}. The saddle point
equation reads $\zeta C'(\zeta)=n$. Using the singular expansion~\eqref{Tsing} of $T(z)$, and a similar expansion for $T'(z)$, this reads
\beq\label{sae1}
\frac{e-5/4}{2\sqrt 2 ({1-\zeta e})^{3/2}}
+\frac 1{4 ({1-\zeta e})}
+O\left( \frac 1{({1-\zeta e})^{1/2}}\right)
=n.
\eeq
This gives  the saddle point as
\beq\label{zeta}
\zeta=\frac 1 e - \frac{(e-5/4)^{2/3}}{2e n^{2/3}}-\frac 1 {6en}
+O(n^{-4/3}).
\eeq
We now want to obtain estimates of the values $A(\zeta)$, $\zeta^n$
and $b(\zeta)$ occurring in Theorem~\ref{thm:saddle}. 
Refining~\eqref{C-b-est} into
\beq\label{C-exp-b}
C(z)=\frac{e-5/4}{\sqrt 2 \sqrt{1-ze}} + \frac 1 4 \log \frac1{2(1-ze)}
+ \frac {19} 8 - \frac {11 e} 3  +O(\sqrt{1-ze}),
\eeq
 we find
$$
C(\zeta)= (e-5/4)^{2/3}n^{1/3} + \frac 1 6 \log \frac n{e-5/4} +\frac
{53} {24} -\frac {11e}3 + O(n^{-1/3}),
$$
which gives
\beq\label{Az}
A(\zeta)\sim \frac{e^{\frac{53} {24} -\frac {11e}3}}{(e-5/4)^{1/6}}
{n^{1/6}}
\exp((e-5/4)^{2/3}n^{1/3}). 
\eeq
It then follows from~\eqref{zeta} that
\beq\label{zeta-n}
\zeta^n \sim e^{-1/6}\exp\left(-n-(e-5/4)^{2/3}\,n^{1/3}/2\right).
\eeq
Finally,
\beq\label{b-b}
b(r)= rC'(r)+r^2C''(r)\sim \frac {3\sqrt 2 (e-5/4)} {8(1-re)^{5/2}}
,
\eeq
so that
$$
b(\zeta) \sim \frac 3 {(e-5/4)^{2/3}}\, n^{5/3}.
$$
Putting  this estimate together with~\eqref{Az} and \eqref{zeta-n}, we
obtain the estimate of $a_n/n!$ given in the proposition.
\end{proof}

\begin{Proposition}[{\bf Number of components --- no bowtie}]
The mean and variance of $N_n$ satisfy:
\[
\E(N_n) \sim (e-5/4)^{2/3} n^{1/3},\quad  \quad
\Var(N_n) \sim \frac 2 3  (e-5/4)^{2/3} n^{1/3},
\]
 and the random variable
$
\frac{N_n-\E(N_n)    }{\sqrt{\Var(N_n) }}
$
converges in law to  a standard normal distribution. 
\end{Proposition}
\begin{proof}
We want to apply Theorem~\ref{thm:ext-adm}. By Proposition~\ref{prop:H1}, $A(z)$ is
H-admissible. Conditions~\eqref{ext-adm0}--\eqref{ext-adm2}
are readily checked, using
$$
C(z) \sim \frac{e-5/4}{\sqrt 2 \sqrt{1-ze}}, \quad 
b(r) \sim  \frac{3\sqrt 2 (e-5/4)}{8(1-ze)^{5/2}} \quad \hbox{and}
\quad V(r) \sim \frac{\sqrt 2 (e-5/4)}{3 \sqrt{1-ze}}.
$$
The asymptotic estimates of $\E(N_n)$ and $\Var(N_n)$ are obtained
through~\eqref{M1M2}, using the saddle point estimate~\eqref{zeta}.
\end{proof}

 Since there are approximately $ n^{1/3}$ components, one may expect
 the size $S_n$ of the root  component to 
be of the order of $ n^{2/3}$. More precisely, we have the following
result.
\begin{Proposition}[\bf{Size of the root  component --- no bowtie}]
\label{prop:number-bowtie}
 The normalized variable $(e-5/4)^{2/3}S_n/(2n^{2/3})$ converges in
distribution to a 
{\rm Gamma}$(3/2,1)$  law, of density $2\sqrt{x}
e^{-x}/\sqrt{\pi}$ on $[0, \infty)$. 
In fact, a local limit law holds: for $x>0$ and $k= \lfloor x\frac  {2n^{2/3}}{(e-5/4)^{2/3}} \rfloor$, 
$$
\frac {2n^{2/3}}{(e-5/4)^{2/3}} \ \PP(S_n=k)\rightarrow 2\sqrt{\frac x \pi} e^{-x}.
$$
 The convergence of moments holds as well: for $i\ge 0$, 
$$
\E(S_n^i) \sim \frac{\Gamma(i+3/2)}{\Gamma(3/2)} \left(\frac{2n^{2/3}}{(e-5/4)^{2/3}}\right)^i.
$$
\end{Proposition}
\begin{proof}
The local (and hence global) limit law follows directly from
Proposition~\ref{prop:b-asympt}, using~\eqref{local}.
For the convergence of the moments, we start from~\eqref{S-moments}. We 
first prove (see Proposition~\ref{prop:H1} in the  appendix) 
that 
$C^{(i+1)}(z)A(z)$ is 
H-admissible. We then apply Theorem~\ref{thm:saddle} to estimate the
coefficient of $z^n$ in this series (we will replace $n$ ny $n-i-1$
later). Our calculations mimic those of
Proposition~\ref{prop:b-asympt}, but the saddle point equation now
reads   
$$
\zeta C'(\zeta)+ \zeta\, \frac{C^{(i+2)}(\zeta)}{C^{(i+1)}(\zeta)} =n,
$$
where $\zeta\equiv \zeta_n^{(i)}$ depends on $i$ and $n$.
Comparing with the original saddle point equation~\eqref{sae1}, and
using the estimate~\eqref{Ci-est} of $C^{(i)}(z)$, this reads
$$
\frac{e-5/4}{2\sqrt 2 ({1-\zeta e})^{3/2}}
+\frac {7+4i}{4 ({1-\zeta e})}
+O\left( \frac 1{({1-\zeta e})^{1/2}}\right)
=n.
$$
This gives  the saddle point as
\beq\label{zeta-i}
\zeta=\frac 1 e - \frac{(e-5/4)^{2/3}}{2e n^{2/3}}-\frac {7+4i} {6en}
+O(n^{-4/3}).
\eeq
We now want to obtain estimates of  $C^{(i+1)}(\zeta)A(\zeta)$, $\zeta^n$
and $b_i(\zeta)$. We first derive from~\eqref{C-exp-b} that
$$
C(\zeta)= (e-5/4)^{2/3}n^{1/3} + \frac 1 6 \log \frac n{e-5/4} +\frac
{29} {24} -\frac {11e}3 -\frac{2i}3+ O(n^{-1/3}).
$$
This gives
\beq\label{Az-i}
A(\zeta)\sim \frac{e^{\frac{53} {24} -\frac {11e}3-\frac{2i}3}}{(e-5/4)^{1/6}}
{n^{1/6}}
\exp((e-5/4)^{2/3}n^{1/3}). 
\eeq
Moreover, we derive from~\eqref{Ci-est} that
\beq\label{Ci+1i}
C^{(i+1)}(\zeta)\sim \frac{(2i+1)!}{2^i i!} e^{i+1}
  \frac{n^{1+2i/3}}{(e-5/4)^{2i/3}}.
\eeq
It then follows from~\eqref{zeta-i} that
\beq\label{zeta-n-i}
\zeta^n \sim e^{-7/6-2i/3}\exp\left(-n-(e-5/4)^{2/3}\,n^{1/3}/2\right).
\eeq
Finally, \eqref{bi} and~\eqref{b-b} give
$$
b_i(\zeta)\sim b(\zeta)
\sim \frac 3 {(e-5/4)^{2/3}}\, n^{5/3}.
$$
Putting  this estimate together with~\eqref{Az-i},~\eqref{Ci+1i} and
\eqref{zeta-n-i}, we obtain 
\begin{multline*}
[z^n]C^{(i+1)}(z)A(z)\sim \\
\frac{(2i+1)!}{2^i i!}\frac{(e-5/4)^{1/6}e^{ {19}/8-{11e}/3}
}{\sqrt {6\pi}}
\frac{e^{n+i+1}}{{ n}^{2/3}} \frac{n^{1+2i/3}}{(e-5/4)^{2i/3}}\,
\exp\left({ \frac 32 (e-5/4)^{2/3} n^{1/3}}\right).
\end{multline*}
We finally replace $n$ by $n-i-1$ (the only effect is to replace
$e^{n+i+1}$ by $e^n$), and divide by the estimate of $n a_n/n!$ given
in Proposition~\ref{prop:b-asympt}: this gives the estimate of the
$i$th moment of 
$S_n$ as stated in the proposition, and concludes the proof.
\end{proof}

 \section{Final comments and further questions}
\label{sec:final}
\subsection{Random generation}
For each of the classes  $\cA$ that we have studied, we have designed an
associated \emm Boltzmann sampler,, which generates a graph $G$ of $\cA$
with probability
\beq\label{boltzmann}
\PP(G)= \frac{x^{|G|}}{|G|!A(x)},
\eeq
where $x>0$ is a fixed parameter such that $A(x)$ converges. We refer
to~\cite[Sec.~4]{boltzmann} for general principles on the construction
of exponential Boltzmann samplers, and only describe how we have addressed
certain specific difficulties. Most of them are related to the fact
that our graphs are unrooted.

\smallskip
\noindent{\bf Trees and forests.} Designing a Boltzmann sampler for \emm
rooted, trees is a basic exercise after reading~\cite{boltzmann}. Note
that the calculation of $T(x)$ can be avoided by  feeding the sampler
directly  with the parameter $t=T(x)$, taken in $(0,1]$. To sample \emm
unrooted, trees, a first solution is to sample a rooted tree $G$ and
keep it with probability $1/|G|$. However, this sometimes
generates large rooted trees that are rejected with high probability. A
 much better solution is presented in~\cite[Sec.~2.2.1]{darrasse}. In
 order to obtain
an unrooted tree distributed according to~\eqref{boltzmann}, one calls
the sampler of rooted trees with a \emm random parameter, $t$. The
density of  $t$ 
must be chosen to be $(1-t)/C(x)$ on $[0,T(x)]$, with $C(x)=T(x)-T(x)^2/2$. To
sample $t$ according to this density, we set
$t=1-\sqrt{1-2uC(x)}$, where $u$ is uniform in $[0,1]$. Again, we actually
avoid computing the series $C(x)$ by feeding directly our sampler with
the value $T(x) \in (0,1]$. We use this trick for all classes that
involve the series $T(x)$.

To obtain large forests (Figure~\ref{fig:forest}), we actually sample
forests with a distinguished vertex; that is, a rooted tree plus a
forest~\cite[Sec.~6.3]{boltzmann}. 

\smallskip
\noindent{\bf Paths, cycles and stars.} The {\tt sequence} operator
of~\cite[Sec.~4]{boltzmann} produces \emm directed, paths, while we need
undirected paths. Let $u$ be uniform on $[0,1]$. Our generator generates the one-vertex 
path if  $u<x/C_p(x)$, where $C_p(x)$ is given by~\eqref{Cp-Cc},
and otherwise generates a path of length $2+\Geom(x)$. An alternative
is to generate a  directed path, and reject it with probability $1/2$
if its size is at least 2.

Although the {\tt cycle} operator of~\cite[Sec.~4]{boltzmann} generates
\emm oriented, cycles, this does not create a similar problem for our
non-oriented cycles: 
indeed, a cycle of length at least 3 has exactly two possible
orientations. 

Designing a Boltzmann sampler $\Gamma RS$ for rooted stars is
elementary. For unrooted stars, we simply call $\Gamma RS$, but reject
the star with probability $1/2$ if it has size 2 (because the only
star with two rootings has size 2).

\smallskip
\noindent{\bf Graphs avoiding the bowtie.} This is the most complex of
our algorithms,  because the generation of connected graphs
involves 7 different cases (see the proof of
Proposition~\ref{prop:bowtie}). There is otherwise no particular
difficulty. We specialize this algorithm to the generation of graphs
avoiding the 2-spoon (Proposition~\ref{prop:2-spoon}). However, the
probability to obtain a forest is about $0.95$, and thus there is no
point in drawing a random graph of this class.

The graphs shown in the paper have been drawn with the {\tt graphviz} software.

\subsection{The nature of the dominant singularities of $\boldsymbol{C(z)}$} 
This is clearly a crucial point, as the probability that $\cG_n$ is
connected and  the quantities $N_n$ and $S_n$  seem to be directly correlated to it (see the summary of our results in
Table~\ref{table}). This raises  the following question: is it possible to describe
an explicit correlation between the properties of  the excluded minors
and the nature of the dominant singularities
of $C(z)$? For instance, it is known that $C(\rho)$ is finite
when all excluded minors are 2-connected, but Section~\ref{sec:trees}
 shows that this happens as well with some non-2-connected excluded
 minors. Which  excluded minors 
give rise to a simple pole in $C(z)$ (as in caterpillars)? or to a logarithmic
singularity (as for graphs with no bowtie nor diamond), or to a
singularity in $(1-z/\rho)^{-1/2}$ (as for graphs with no  bowtie)?

Some classes for which $C(z)$ has a unique dominant pole of high order
are described in the next subsection.

\subsection{More examples and predictions}
Our examples, as well as a quick analysis,  lead us to predict the
following results when $C(z)$ has a unique 
dominant singularity and a singular behaviour of the form
$(1-z/\rho)^{-\alpha}$, with $\alpha >0$:
\begin{itemize}
\item the mean and variance of the number $N_n$ of components scale like
  $n^{\alpha/(1+\alpha)}$, and $N_n$ admits a gaussian limit law after
  normalization,
\item the mean of $S_n$ scales like  $n^{1/(1+\alpha)}$, and $S_n$,
  normalized by its expectation, converges to a Gamma distribution of
  parameters $\alpha+1$ and $1$.
\end{itemize}
The second point is developed in~\cite{porte}. To confirm these predictions one could  study the following
classes, which yield series $C(z)$ with a high order dominant
pole. Fix $k\ge 2$, and consider the class $\cA^{(k)}$ of forests of
degree at most $k$, in which each
component has at most one vertex of degree $\ge 3$.  This means that
the components are stars with long rays and ``centers'' of degree at
most $k$.  It is not hard to see that
$$
C^{(k)}(z)= z +\frac{z^2}{2(1-z)} + \sum_{i=3}^k \frac{z^{i+1}}{i!
  (1-z)^i},
$$
so that $C_k$ has a pole of order $k$ (for $k\ge 3$). The case $k=2$ corresponds to
forests of paths (Section~\ref{sec:pole}).
The limit case $k=\infty$ (forests of stars with long rays) looks
interesting, with a very fast divergence of $C$ at 1:
$$
C^{(\infty)}(z)= z \exp \left( \frac z{1-z}\right) -\frac{z^2}{2(1-z)^2}.
$$
 We do not dare any prediction here.

\subsection{Other parameters}
We have focussed in this paper on certain parameters that are well
understood when all excluded minors are 2-connected. But other parameters ---  number of edges,  size of the
largest 2-connected component,  distribution of vertex
degrees --- have
been investigated in other contexts, that sometimes intersect the study of minor-closed classes~\cite{MR2534261,MR2538796,fusy-subcritical,drmota-degree,gimenez-given}. When
specialized to the theory of minor-closed classes, these papers
generally assume that all excluded minors are 2-connected, sometimes even
3-connected. 

Clearly, it would not be hard to keep track of the number
of edges in our enumerative results. Presumably,
keeping track of the number of vertices of degree $d$ for any (fixed)
$d$ would not be too difficult either. This may be the topic of future work.
The size of the largest component clearly needs a further
investigation as well.

\bigskip
\noindent {\bf Acknowledgements.}  KW would like to thank Colin McDiarmid for many inspiring and helpful
discussions as well as for constant support. We also thank  Philippe Duchon
and Carine Pivoteau for their help with Boltzmann samplers, Nicolas
Bonichon for a crash course on the Graphviz software,  Jean-François
Marckert for pointing out the relevance of~\cite{arratia} and
finally Bruno Salvy
for pointing out Reference~\cite{macintyre}.

We also thank the referees for their careful and detailed reports.

\spacebreak
\bigskip 
\begin{center}{\large{\bf Appendix: Hayman-admissibility for bowties}}
\end{center}

\medskip   

\begin{Proposition}
  \label{prop:H1}
Let $C(z)$ and $A(z)$ be the series given in
Proposition~\ref{prop:bowtie}. Then the series $A(z)$ and
$C^{(i+1)}(z)A(z)$ are H-admissible, for any $i\ge 1$.
\end{Proposition}
\begin{proof}
  We begin with the series $A(z)$. Recall the analytic properties of
  $T(z)$, listed at the beginning of the proof of
  Proposition~\ref{prop:b-asympt}. The capture condition ${\bf H_1}$
  is readily checked. In fact,
\beq\label{b-est}
a(r)=rC'(r)  \sim \frac{e-5/4}{2\sqrt 2 ({1-re})^{3/2}} \quad \hbox{and} \quad 
b(r)= rC'(r)+r^2C''(r) \sim \frac{3\sqrt 2(e-5/4)}{8 (1-re)^{5/2}}.
\eeq

Let us now prove ${\bf H_2}$. By Taylor's formula applied to the
function $f:\theta\mapsto C(re^{i\theta})$, we have, for $r \in (0,1/e)$ and
$\theta \in [-\theta_0, \theta_0]$:
$$
\left| C(re^{i\theta}) -C(r) -i\theta a(r) + \theta^2 b(r)/2\right|
\le |\theta_0^3|/6 \sup_{|\alpha| \le \theta_0} |f^{(3)}(\alpha)|
$$
with
\begin{eqnarray*}
  \left|f^{(3)}(\alpha)\right|&=&
\left|-ir e^{i\al} C'(r e^{i\al}) -3ir^2 e^{2i\al} C''(r e^{i\al}) -ir^3
e^{3i\al} C'''(r e^{i\al}) \right|
\\
&\le& r  C'(r ) +3r^2 C''(r) +r^3
 C'''(r) \sim \frac \kappa{(1-re)^{7/2}}
\end{eqnarray*}
as $r\rightarrow 1/e$, for some constant $\kappa$. Hence, if we take
$\theta_0 \equiv \theta_0(r) =o((1-re)^{7/6})$, then 
$$
\sup_{|\theta |\le \theta_0(r)}\left| \frac{A(re^{i\theta})}{A(r)}
  e^{-i\theta a(r)+\theta^2b(r)/2}-1\right|
= \left|e^{o(1)}-1\right| \rightarrow 0
$$
as $r\rightarrow 1/e$. Thus ${\bf H_2}$ holds for such values of
$\theta_0$. We now take 
\beq\label{t0-b}
 \theta_0(r) =(1-re)^{6/5},
\eeq
and want to prove that ${\bf H_3}$ also holds.

Recall that $C(z)$ is analytic on $\cs\setminus [1/e, \infty)$, and
let us isolate in $C(z)$ the part that diverges at $z=1/e$:
\beq\label{C-bowtie-exp}
C(z)=   \frac{c}{ \sqrt{1-ze}} + \frac 1 4 \log \frac 1
{1-ze} + O(1) 
\eeq
where
$
c= ({e-5/4})/{\sqrt 2}>0 . 
$
It follows that 
$$
B(z):= C(z)- \frac{c}{\sqrt{1-ze}} - \frac 1 4 \log \frac 1
{1-ze}
$$
is uniformly bounded on $\{|z| < 1/e\}$.
Hence, writing $z=re^{i\theta}$, we have
$$
\sup_{|\theta|\in [\theta_0, \pi)} \left| \frac{A(z)}{A(r)}
  \sqrt{b(r)}
\right|
\le M \sup_{|\theta|\in [\theta_0, \pi)} \left|
\frac{1-re}{1-ze}\right|^{1/4} \left| \exp\left( \frac c{\sqrt{1-ze}} -
  \frac c{\sqrt{1-re}}\right) \right| \sqrt{b(r)}
$$
for some constant $M$.

For any $z$ of modulus $r<1/e$, we have $|1-re|\le |1-ze|$, and we can
bound the first factor above by 1. Also, it is not hard to prove that 
$\Re \left(1/\sqrt{1-ze}\right)$ is a decreasing function of $\theta \in (0,
\pi)$.
Hence, denoting $z_0= r e^{i\theta_0}$:
$$
\sup_{|\theta|\in [\theta_0, \pi)} \left| \frac{A(z)}{A(r)}
  \sqrt{b(r)}
\right|\le M \exp\left( \Re\left(\frac c{\sqrt{1-z_0e}} \right)-
  \frac c{\sqrt{1-re}}\right) \sqrt{b(r)}.
$$
But as $r\rightarrow 1/e$, the choice~\eqref{t0-b} of $\theta_0$
implies that
$$
\Re\left(\frac c{\sqrt{1-z_0e}} \right)-   \frac c{\sqrt{1-re}}
= -\frac {3c} {8(1-re)^{1/10}} +o(1).
$$
Condition  ${\bf H_3}$ now follows, using the estimate~\eqref{b-est}
of $b(r)$.

\medskip
Let us now consider the series $A_i(z):=C^{(i+1)}(z) A(z)$, for $i\ge 1$. It
is easy to prove by induction on $i$ that for $i\ge 1$,
\beq\label{Ci-est}
C^{(i)}(z)= \frac{(2i)!(e-5/4) e^i}{4^i \sqrt 2 i! (1-ze)^{i+1/2}}+
O\left( \frac1{(1-ze)^i}\right).
\eeq
This can be proved either from the expression of $C(z)$ given in
Proposition~\ref{prop:bowtie}, or 
by starting from the singular expansion~\eqref{C-bowtie-exp} of $C(z)$
and applying~\cite[Thm.~VI.8, p.~419]{flajolet-sedgewick}.

Recall the behaviour~\eqref{b-est} of the functions $a(r)$ and $b(r)$ associated
with $A(z)$. It follows, with obvious notation, 
that   as $r\rightarrow 1/e$,
\beq\label{ai}
a_i(r) =a(r)
+r \frac{C^{(i+2)}(r)}{C^{(i+1)}(r)} = a(r)+O\left(\frac 1{1-re}\right)
\eeq
{and}
\beq\label{bi}
b_i(r)= b(r)
+ r \frac{C^{(i+2)}(r)}{C^{(i+1)}(r)}
+ r^2\frac{C^{(i+3)}(r)}{C^{(i+1)}(r)}-
r^2\left(\frac{C^{(i+2)}(r)}{C^{(i+1)}(r)}\right)^2
= b(r)+O\left(\frac 1{(1-re)^2}\right)
\eeq
both tend to infinity. Thus
${\bf H_1}$ holds.

Let us now prove that $A_i(z)$ satisfies ${\bf H_2}$ with the same value of
$\theta_0$ as  for $A(z)$ (that is, $\theta_0=(1-re)^{6/5}$). 
Thanks to~(\ref{ai}--\ref{bi}) we have, for $|\theta|\le \theta_0$ and
uniformly in $\theta$,
$$
 e^{-i\theta a_i(r)+\theta^2b_i(r)/2}= e^{-i\theta
   a(r)+\theta^2b(r)/2}\left(1+O((1-re)^{1/5})\right).
$$
Now using~\eqref{Ci-est}, we obtain, denoting $z=re^{i\theta}$,
$$
\frac{C^{(i+1)}(z)}{C^{(i+1)}(r)}=\left(\frac{1-ze}{1-re}\right)^{-i-3/2}\left(1+O((1-re)^{1/2})\right)
=1+O((1-re)^{1/5}).
$$
Hence
$$\frac{A_i(z)}{A_i(r)} e^{-i\theta a_i(r)+\theta^2 b_i(r)/2}=
\frac{A(z)}{A(r)} e^{-i\theta
  a(r)+\theta^2b(r)/2}\left(1+O((1-re)^{1/5})\right)$$
and Condition ${\bf H_2}$ holds for $A_i$ since it holds for $A$.

Finally, since $C(z)$ has non-negative coefficients, we have
$\left|C^{(i+1)}(z)\right| \le C^{(i+1)}(r) $ for
$z=re^{i\theta}$. Thus the fact that $A_i(z)$ satisfies ${\bf H_3}$
follows from the fact that $A(z)$ satisfies ${\bf H_3}$, together with
$b_i(r)\sim b(r)$.
\end{proof}

\bibliographystyle{plain}
\bibliography{minor.bib}

\end{document}